\documentclass[11pt,reqno]{amsart}

\usepackage{amsmath, amsfonts, amsthm, amssymb, stmaryrd, color, cite}

\usepackage{hyperref}
\newtheorem{theorem}{Theorem}[section]
\newtheorem{lemma}{Lemma}[section]
\newtheorem{proposition}{Proposition}[section]

\theoremstyle{definition}
\newtheorem{definition}{Definition}[section]

\theoremstyle{remark}

\numberwithin{equation}{section}
\allowdisplaybreaks

\def\f{\frac}

\def\hf1{^\f{1}{1-\xi^2}}

\def\be{\begin{equation}}
\def\en{\end{equation}}
\def\bs{\begin{split}}
\def\es{\end{split}}
\def\ba{\begin{align}}
\def\ea{\end{align}}

\author[P. Fatheddin]{Parisa Fatheddin}
\address{Department of Mathematics, University of Pittsburgh, Pittsburgh, 15260, USA.}
\email{PAF49@pitt.edu,}

\author[Z. Qiu]{Zhaoyang Qiu}
\address{School of Mathematics and Statistics, Huazhong University of Science and Technology, Wuhan, 430074, China.}
\email{ZHQMATH@163.com }
\title[\emph{Large Deviations for Stochastic Schr\"odinger Equation}]
{\Large{Large Deviations for Nonlinear Stochastic Schr\"odinger Equation}}

\title[\emph{Large Deviations for Stochastic Schr\"odinger Equation}]
{\Large{Large Deviations for Nonlinear Stochastic Schr\"odinger Equation}}

\date{}

\begin{document}
\begin{abstract}
Large deviation principle by the weak convergence approach is established for the stochastic nonlinear Schr\"{o}dinger equation in one-dimension and as an application the exit problem is investigated.
\end{abstract}

\maketitle
\section{Introduction}

The Schr\"odinger equation describes how the quantum state of the physical mechanics changes with time, which is widely used to model many important physical phenomena, such as Taylor-Couette flow, superconductivity, and Bose-Einstein condensation. It is used especially in optics to model laser beam propagation through random media. For more background, we refer the reader to \cite{Andrews, Sulem}.

Because of its importance in applications in physics, the Schr\"odinger equation has been broadly studied. Here we consider the stochastic nonlinear Schr\"odinger equation of the form,
\begin{eqnarray}\label{Schrodinger}
du(t)= i\Delta u(t)dt - \lambda |u(t)|^{2\sigma}u(t)dt +g(t,u(t))dW,
\end{eqnarray}
where $\sigma\in [1,p]$ for all $p> 1$, $\lambda >0$ and $W$ is a Q-Wiener process to be made precise in the subsequent section. Many authors have studied the well-posedness of the  stochastic Schr\"{o}dinger equation which takes several forms depending on the physical situation.   A. Bouard and A. Debussche \cite{Bouardmultiplicative} established the well-posedness of weak solution to equation (\ref{Gautier}) given below with $\sigma=1$ in the case of multiplicative noise using the fixed point argument. In addition, in \cite{BouardH1} they enhanced the result to the energy space $H^1$  in both cases of additive and multiplicative noise. Using the Galerkin approximation coupled with the energy method, W. Grecksch and H. Lisei\cite{Lisei1} proved the variational solution to equation (\ref{Schrodinger}) with the second and third terms on the right hand side multiplied by $-i$ and $i$, respectively, for all $\sigma\geq 1$ and $W$ being the cylindrical Wiener process and also they extended their result to the system driven by both cylindrical Wiener process and fractional Brownian motion. Furthermore, V. Barbu, M. R\"{o}ckner and D. Zhang \cite{Barbu1} established the well-posedness of weak solution to equation,
\begin{eqnarray*}
idu(t)=\Delta u(t) dt +i\mu u(t) dt-\lambda|u(t)|^{\sigma-1} u(t)dt +i u(t)dW.
\end{eqnarray*}
The constitution of linear multiplicative noise makes it possible to reduce the system to a random nonlinear equation with lower-order terms by a rescaling transform, then results in the deterministic setting may be applied directly. As a natural continuation of \cite{Barbu1}, they obtain the well-posedness of solution in the energy space $H^1$ in \cite{Barbu2}.

We achieve the large deviation principle by the weak convergence approach introduced by \cite{BDM,Budhiraja} for (\ref{Schrodinger}). Large deviations is a growing field with applications in communications, finance and statistical mechanics. For more details see \cite{Ellis2, Ganesh, Glasserman}. Regarding the Schr\"odinger equation there were some pioneer works on large deviation principle using the classical method achieved by E. Gautier for the equation,   \begin{equation}\label{Gautier}
 idu(t)= \Delta u(t)dt+\lambda |u(t)|^{2\sigma} u(t) dt +(dW \text{ or } u(t)dW \text{ or } dW^{H}),
\end{equation}
where additive noise was considered in \cite{additive}, linear multiplicative noise in \cite{multiplicative} and fractional noise in \cite{fractional}.  More specifically, the Azencott method was implemented in the case of multiplicative noise by proving the continuity of the controlled equation and achieving the Freidlin-Wentzell inequality. For additive and fractional noise, E. Gautier relies on large deviations already known for centered Gaussian measures and applies the Dawson-G\"{a}rtner theorem. The weak convergence approach implemented here, does not involve technical exponential tail estimates with time discretization making the proof more concise. For some results on large deviation principle by the weak convergence approach see \cite{Bessaih, Ss, Zhai} for stochastic Navier-Stokes equation, \cite{Duan,ours} for the stochastic Boussinesq equation, and \cite{shell} for stochastic shell model.

Before giving the proof of the large deviation principle, we follow the one dimensional setting of \cite{Lisei3} to show the well-posedness of stochastic controlled Schr\"odinger equation directly. Since the noise term is not uniformly bounded in $\epsilon$, the Girsanov transformation theorem can not be applied to obtain this well-posedness result.  Unlike most fluid models, due to the complex nature of the Sc\"{o}dinger equation, we can only achieve an a priori $L^2$-bound, and a higher regularity of the approximate solution cannot be obtained when the It\^{o} formula is applied. Therefore, in order to achieve higher regularity estimates of the approximate solution, we have to impose the $H^1$ regularity on the initial data. Also because of the lack of the smoothing properties of the Schr\"odinger equation, space time white noise cannot be used and we are limited to noise that is white in time and colored in space.

Similar to the result in \cite{shell}, we are only able to prove large deviations on a weaker space $\mathcal{C}([0,T];H)$ than the space of solutions, $\mathcal{C}([0,T];H) \cap L^{p}(0,T;V)$, due to the cancellation of the diffusion term, $i\Delta u$ when the It\^{o} formula is applied. To achieve the compactness in space $\mathcal{C}([0,T];H)$, the convergence, $u^\epsilon\rightarrow u,~ {\rm in }~L^{2}(0,T;H), ~\mathbb{P} ~\mbox{a.s.} $
as $\epsilon\rightarrow 0$, is established in advance in the spirit of \cite{wang}, without using the time discretization estimates developed by \cite{Duan}. Different from \cite{wang}, thanks to Gy\"{o}ngy-Krylov's lemma and uniqueness of solutions, we are able to recover the convergence $u^\epsilon\rightarrow u$ on the original probability space, which makes the proof more direct. Also, we give a detailed proof that the limit $u$ is the solution of the controlled system following the idea of \cite{Bessaih}.

Exit problem may also be considered as a consequence of large deviations as was proved by E. Gautier in \cite{Gautierexit}. Here the probability of exit time being before a given time, $T_{0}$ is studied using the large deviations result by a different method following results by M. Freidlin, A. Wentzell in Chapters 3 and 4 of \cite{Freidlin}.

The article is organized as follows. We begin in Section 2 with notations, deterministic and stochastic preliminaries needed for the rest of the paper and state the main results. Then we focus in Section 3 on establishing the well-posedness of the controlled system. With this in hand, the large deviation principle is achieved in Section 4 and in Section 5, the exit problem is shown. We have also included an Appendix stating the results used frequently in the paper.

\section{Preliminaries and Main Results}\label{sec2}
Let $H$ denote the space $L^{2}(\mathcal{D})$ and $V$ be the space $H^{1}(\mathcal{D})$ with norms,
\begin{eqnarray*}
&&\|u\|^2:=\int_{\mathcal{D}}u\cdot \bar{u}dx,~{\rm for}~ u\in H,\\
&&\|u\|^2_V:=\int_{\mathcal{D}}u\cdot \bar{u}+\partial_{x}u\cdot \partial_{x}\bar{u}dx,~{\rm for}~ u\in V,
\end{eqnarray*}
respectively, where $\mathcal{D}=(0,1)$ is the domain.  Using notation $X^{*}$ as the dual space of $X$, we apply operator $A: V\rightarrow V^{*}$ defined as $\langle Au,v\rangle= \int_{\mathcal{D}} \frac{d}{dx}u(x)\frac{d}{dx}\bar{v}(x)dx$ for $u,v\in V$, to write $Au=-\Delta u$. Furthermore, for $u\in V$, letting $f(u):= |u|^{2\sigma}u$, equation \eqref{Schrodinger} may be given in the following abstract form,
\begin{eqnarray}\label{2.1}
du(t)+iAu(t)dt=\lambda f(u(t))dt + g(t, u(t))dW,
\end{eqnarray}
where $W$ is an $H$-valued Wiener process defined on a probability space $(\Omega, \mathcal{F}, \{\mathcal{F}_t\}_{t\geq 0}, \mathbb{P})$, with the covariance operator $Q$, adapted to the complete, right continuous filtration $\{\mathcal{F}_t\}_{t\geq 0}$. If $\{e_{k}\}_{k\geq 1}$ is a complete orthonormal basis of $H$ such that $Qe_{i}=\lambda_{i}e_{i}$, then $W$ can be written formally as the expansion $W(t,\omega)=\sum\limits_{k\geq 1}\sqrt{\lambda_{k}}e_{k}W_{k}(t,\omega)$ where $\{W_{k}\}_{k\geq 1}$ is a sequence of independent standard one-dimensional Brownian motions. We also have that $W\in \mathcal{C}([0,\infty),H)$ almost surely, see \cite{Prato}.

Let $H_{0}=Q^{\frac{1}{2}}H$, then $H_{0}$ is a Hilbert space with inner product,
\begin{equation*}
( h,g)_{H_{0}}=(Q^{-\frac{1}{2}}h,Q^{-\frac{1}{2}}g)_{H},~ \forall~~ h,g\in H_{0},
\end{equation*}
 and the induced norm, $\|\cdot\|_{H_{0}}^{2}=(\cdot,\cdot)_{H_{0}}$. The imbedding map $i:H_{0}\rightarrow H$ is Hilbert-Schmidt and hence compact operator with $ii^{\ast}=Q$. Now consider another separable Hilbert space $X$ and let $L_{Q}(H_{0},X)$ be the space of linear operators $S:H_{0}\rightarrow X$ such that $SQ^{\frac{1}{2}}$ is a linear Hilbert-Schmidt
operator from $H$ to $X$,  endowed with the norm $\|S\|_{L_{Q}}^{2}=tr(SQS^{\ast}) =\sum\limits_{k}| SQ^{\frac{1}{2}}e_{k}|_{X}^{2}$. Set $L_2(H_0,X)=\{SQ^{\frac{1}{2}}: \, S\in L_{Q}(H_{0},X)\}$ with the norm defined by $\|f\|^2_{L_2(H_0,X)}  =\sum\limits_{k}|fQ^\frac{1}{2}e_{k}|_{X}^{2}$.

For an $X$-valued predictable process $G\in L^{2}(\Omega;L^{2}_{loc}([0,\infty),L_{2}(H_0,X)))$, by taking $G_{k}=GQ^\frac{1}{2}e_{k}$, one can define the stochastic integral,
\begin{equation*}
M_{t}:=\int_{0}^{t}GdW=\sum_{k}\int_{0}^{t}
GQ^\frac{1}{2}e_{k}dW_{k}=\sum_{k}\int_{0}^{t}G_{k}dW_{k},
\end{equation*}
as an element in $\mathcal{M}_{X}^{2}$ which is the space of all $X$-valued square integrable martingales. For more details see \cite{Prato}. For the process $\{M_{t}\}_{t\geq 0}$, the Burkholder-Davis-Gundy inequality holds, which in the present context takes the form,
\begin{eqnarray}\label{2.2}
\mathbb{E}\left[\sup_{t\in [0,T]}\left\|\int_{0}^{t}GdW\right\|_{X}^{p}\right]\leq c_{p}\mathbb{E}\left(\int_{0}^{T}\|G\|_{L_{2}(H_0,X)}^{2}ds\right)^{\frac{p}{2}},
\end{eqnarray}
for any $p\geq1$.

Based on Lemmas 7.1 and 7.2 in \cite{Lisei3}, function $f:V\rightarrow H$ has the following properties for every $u,v\in V$, which will be used throughout the rest of the paper,
\begin{eqnarray}
&&\text{Re}(f(u),u) \geq 0, \hspace{.3cm} \text{and} \hspace{.3cm} \text{Re} (f(u)-f(v), u-v) \geq 0,\label{2.3}\\
&&\|f(v)\|^{2} \leq 2^{2\sigma +1} \|v\|_{V}^{4\sigma +2},\label{2.4}\\
&&\|f(u)-f(v)\|^2\leq 2^{2\sigma +1} (4\sigma -1)^{2}(\|u\|_V^{4\sigma}+\|v\|_V^{4\sigma})\|u-v\|^2,\label{2.5}
\label{fsigma}\\
&&{\rm Re} (f(v),Av)\geq 0 \label{extra}.
\end{eqnarray}
Moreover, we assume that the operator $g$ satisfies the Lipschitz continuous and linear growth conditions. Namely,
\begin{eqnarray}
&&\|g(t,u)\|_{L_{2}(H_0,H)}^{2}\leq K_{1}(1+ \|u\|^{2}), ~{\rm for ~ all}~ u\in H,\label{2.6}\\
&&\|g(t,u)\|_{L_{2}(H_0,V)}^{2}\leq K_{1}(1+ \|u\|_V^{2}), ~{\rm for ~ all}~ u\in V,\label{2.7}\\
&&\|g(t,u)-g(t,v)\|_{L_{2}(H_0,H)}^{2}\leq K_{2}\|u-v\|^{2},  ~{\rm for ~ all}~ u, v \in H. \label{2.8}
\end{eqnarray}

For large deviations, we consider variational solutions which are equivalent to weak solutions in the PDE sense but strong solutions in the probability sense, to the following stochastic controlled equation,
\begin{equation}\label{original}
u^{\epsilon}(t)= u_0-i\int_{0}^{t}Au^{\epsilon}(s)ds - \lambda \int_{0}^{t} f(u^{\epsilon}(s))ds+ \sqrt{\epsilon}\int_{0}^{t} g(s, u^{\epsilon}(s))dW+\int_{0}^{t} g(s, u^{\epsilon}(s))h(s)ds,
\end{equation}
and to the skeleton equation given by,
\begin{eqnarray}\label{controlled}
u_{h}(t)= u_{0}-i\int_{0}^{t} Au_{h}(s)ds -\lambda \int_{0}^{t}f(u_{h}(s))ds +\int_{0}^{t}g(s, u_{h}(s))h(s)ds,
\end{eqnarray}
where $h\in \mathcal{S}_{M}(H_0):= \left\{u\in L^{2}(0,T;H_0): \int_{0}^{T} \|u(s)\|_{H_0}^{2}ds \leq M\right\}$, which is a Polish space endowed with the weak topology $d(h,g)=\sum_{k\geq 1}\frac{1}{2^k}\left|\int_0^t(h(s)-g(s), \xi_k) ds\right|$, with the sequence $\{\xi_k\}_{k\geq 1}$ being an orthonormal basis of $L^{2}(0,T;H_0)$.

First, we have the following well-posedness result to equation (\ref{original}). Similarly we may prove the well-posedness of the variational solution to \eqref{controlled}, denoted as $\mathcal{G}(\int_{0}^{\cdot} h(s)ds)$.
\begin{theorem}\label{theorem1}
Suppose that the initial data $u_0$ is $\mathcal{F}_0$ measurable random variable, and is in $L^{p}(\Omega; V)$ for any $p\geq 1$ and operator $g$ satisfies conditions (\ref{2.6})-(\ref{2.8}).  For any $\epsilon\in [0,1]$, there exists a unique solution of Schr\"odinger equation \eqref{original} in space $\mathcal{C}([0,T];H)\cap L^{p}(0,T;V)$, $\mathbb{P}$ a.s. in the following sense,
\begin{eqnarray*}
 (u^{\epsilon}(t),\phi)=  (u_0,\phi)-i\int_{0}^{t}(\nabla u^{\epsilon}(s),\nabla \phi) ds - \lambda \int_{0}^{t}  (f(u^{\epsilon}(s)), \phi) ds\nonumber \\+ \sqrt{\epsilon}\int_{0}^{t}( g(s, u^{\epsilon}(s)), \phi) dW+\int_{0}^{t}  (g(s, u^{\epsilon}(s))h(s),\phi) ds,
\end{eqnarray*}
for any $\phi\in V$ and $t\in [0,T]$. Moreover, the following energy estimates hold for all $p\geq 1$,
\begin{eqnarray}
&&\sup_{\epsilon\in [0,1]}\mathbb{E}\left[ \sup_{t\in [0, T]} \|u^{\epsilon}(t)\|^{p} \right]\leq C\left( p, T, K_{1}, \mathbb{E}\|u_0\|^{p}\right), \label{2.12}\\
&&\sup_{\epsilon\in [0,1]}\mathbb{E} \int_{0}^{T} \|u^{\epsilon}(t)\|_{V}^{p}dt \leq C\left( p, T, K_{1}, \mathbb{E}\|u_0\|_{V}^{p}\right).\label{2.13}
\end{eqnarray}

\end{theorem}
Using the above notation, we have the following large deviations result by applying theorem 6 of \cite{BDM}.

\begin{theorem}
The family $\{u^{\epsilon}(\cdot)\}_{\epsilon\in [0,1]}$ satisfies the large deviation principle in $\mathcal{C}([0,T];H)$ with rate function,
\begin{equation*}
I(v)=
\frac{1}{2}\inf_{\{h\in L^2(0,T; H_0): v=\mathcal{G}(\int_{0}^{\cdot} h(s)ds)\}}\int_{0}^{T} \|h(s)\|_{H_0}^{2}ds,
\end{equation*}
where the infimum of empty set is taken to be infinity.
\end{theorem}

As a direct consequence of large deviations, we obtain the result below concerning the exit problem with,
\begin{eqnarray}\label{tau}
\tau^{\epsilon}:= \inf \left\{t: u^\epsilon(t)\in D_{T_{0}}^c\right\},
\end{eqnarray}
where $D_{T_{0}}$ is defined by $D_{T_{0}}=\left\{u^\epsilon: \sup_{t\in [0,T_0]}\|u^\epsilon(t)\|^{2}\leq r \right\}$ for any $r\in \mathbb{R}^+$.

\begin{theorem}
For any given time $T_{0}$ and fixed $\delta>0$, $\epsilon < \frac{1}{K_{1}T_{0}}$, the probability of exit time being before $T_{0}$ can be estimated by,
\begin{eqnarray}\label{exit}
\exp\left(-\frac{1}{\epsilon}\left(\inf_{v\in D_{T_{0}}^{c}}I(v)+\delta\right)\right) \leq \mathbb{P}(\tau^{\epsilon}\leq T_0)
\leq  \frac{2 \sqrt{\epsilon K_{1}}C(T_0, u_{0})}{r\left(1-\epsilon K_{1}T_{0}\right)- \|u_0\|^{2} -\epsilon K_{1}T_{0}}.
\end{eqnarray}
Also the mean exit time is bounded as follows,
\begin{eqnarray*}\label{mean}
\mathbb{E}(\tau^{\epsilon})\leq \frac{1}{1-\exp\left(-\frac{1}{\epsilon}\left(\inf_{v\in D_{1}}I(v)-\delta\right)\right)}.
\end{eqnarray*}
\end{theorem}

\section{well-posedness}
In this section, we consider the well-posedness of the equations (\ref{original}) and (\ref{controlled}), which will be divided into two parts. For the first part, we aim to prove the existence of solutions by implementing the Galerkin approximation scheme following the method applied in \cite{Breckner, Lisei3}, where we refer to Theorems 5.1 and 5.2 in \cite{Lisei3} for some of the details to avoid repetition. For the second part, we show the uniqueness of solutions.

Define $H_{n}= \text{span}\{e_{1}, e_{2},..., e_{n}\}$ and let $\pi_{n}$ be the projector from $H$ to  $H_{n}$ given by,
\begin{eqnarray*}
\pi_{n}u= \sum_{k=1}^{n}(u,e_{k})e_{k},~{\rm for} ~u\in H.
\end{eqnarray*}
To handle the nonlinear term, we consider for each fixed $R\in \mathbb{N}$,
\begin{eqnarray*}
 \Psi^{R}(r) = \left\{ \begin{array}{ll}
         1, & 0\leq r\leq R,\\
         R+1-r,& R<r<R+1,\\
         0, & r\geq R+1, \\
 \end{array} \right.
\end{eqnarray*}
then $\Psi^{R}(\|u_{n}^{\epsilon,R}\|)f_{n}(u^{\epsilon, R}_{n})$ is lipschitz continuous and satisfies the linear growth condition, which along with conditions (\ref{2.6}), (\ref{2.8}) imply the existence and uniqueness of solutions in $L^{2}(\Omega; \mathcal{C}([0,T];H_{n}))$ to the finite dimensional system,
\begin{eqnarray}\label{3.1}
d(u^{\epsilon, R}_{n}, \varphi)&=& - i \langle Au^{\epsilon, R}_{n}, \varphi\rangle dt - \lambda \left(\Psi^{R}(\|u^{\epsilon, R}_{n}\|)f_{n}(u^{\epsilon, R}_{n}), \varphi\right)dt\nonumber \\
&+&\sqrt{\epsilon}\left(g_{n}(t, u^{\epsilon, R}_{n}(t))dW, \varphi\right)+\left(g_{n}(t,u^{\epsilon, R}_{n}(t))h, \varphi\right)dt,
\end{eqnarray}
for $\varphi\in H_{n}$. We next achieve the required a priori estimates to solution $u^{\epsilon,R}_{n}$ of the finite dimensional system (\ref{3.1}).
\begin{lemma}  For any $n\geq 1$, $\epsilon\in [0,1]$ and $p\geq 1$, assume that the initial data $u_0$ is $\mathcal{F}_0$ measurable random variable with $u_0\in L^p(\Omega; V)$ and conditions \eqref{2.6}, \eqref{2.7} hold, then there exists constant $C=C(M,T,p,K_1, \mathbb{E}\|u_0\|_V^p)$ independent of $n, \epsilon$ and $R$ such that
\begin{eqnarray}\label{3.2}
\sup_{\epsilon\in [0,1]}\left(\mathbb{E} \left[\sup_{t\in [0,T]} \|u_{n}^{\epsilon,R}(t)\|^p\right]+\mathbb{E} \int_{0}^{T}  \|u_{n}^{\epsilon,R}(t)\|_V^{p}dt\right) \leq C.
\end{eqnarray}
\end{lemma}
\begin{proof}Applying the It\^{o} formula to $\|u_{n}^{\epsilon, R}\|^{p}$ gives,
\begin{eqnarray*}
d\|u_{n}^{\epsilon,R}(t)\|^{p} &=& p \text{Im} \langle Au_{n}^{\epsilon,R}(t), u_{n}^{\epsilon,R}(t)\rangle\|u_{n}^{\epsilon,R}(t)\|^{p-2}dt\\
 &&-\lambda p  \|u_{n}^{\epsilon,R}(t)\|^{p-2}\text{Re}  \left(\Psi^{R}(\|u_{n}^{\epsilon,R}(t)\|)f(u_{n}^{\epsilon,R}(t)), u_{n}^{\epsilon,R}(t)\right)dt\\
&&+ p \sqrt{\epsilon} \text{Re} \left(g_{n}(t,u_{n}^{\epsilon,R}(t)), u_{n}^{\epsilon,R}(t)\right)\|u_{n}^{\epsilon,R}(t)\|^{p-2}dW\\
&& + p\text{Re}\left(g_{n}(t, u_{n}^{\epsilon,R}(t))h(t), u_{n}^{\epsilon,R}(t)\right)\|u_{n}^{\epsilon,R}(t)\|^{p-2}dt\\
&&+ \epsilon p \|g_{n}(t,u_{n}^{\epsilon,R}(t))\|_{L_{2}(H_{0},H)}^{2}\|u_{n}^{\epsilon,R}(t)\|^{p-2}dt\\
&& +\epsilon \frac{p(p-2)}{2} \left|{\rm Re}\left(g_{n}(t,u_{n}^{\epsilon,R}(t)), u_{n}^{\epsilon,R}(t)\right)\right|^{2} \|u_{n}^{\epsilon,R}(t)\|^{p-4}dt.
\end{eqnarray*}
Here, we focus on the fourth term on the right hand side since estimates for bounds by constants for the rest of the terms were shown in the proof of Theorem 4.1 of \cite{Lisei3}. Let
\begin{eqnarray}
\tau_{N}:= \inf\left\{t:\|u_{n}^{\epsilon,R}(t)\|^{p} > N\right\},
\end{eqnarray}
if the set is empty, take $\tau_{N}=T$. Note that $\tau_{N}$ is increasing with $\lim_{N\rightarrow\infty}\tau_{N}=T$, $\mathbb{P}$ a.s.
Taking the integral on $[0,t\wedge \tau_{N}]$, then expectation, using the Cauchy-Schwarz inequality along with condition \eqref{2.6}, the fourth term may be bounded as follows,
\begin{eqnarray*}\label{estimate1}
&&\mathbb{E}\left|p \int_{0}^{t\wedge \tau_{N}}\text{Re}  \left(g_{n}(s,u_{n}^{\epsilon,R}(s))h(s), u_{n}^{\epsilon,R}(s)\right)\|u_{n}^{\epsilon,R}(s)\|^{p-2}ds\right|\nonumber\\
&\leq& p\mathbb{E}\int_{0}^{t\wedge \tau_{N}}\|u_{n}^{\epsilon,R}(s)\|^{p-1} \|g_{n}(s,u_{n}^{\epsilon,R}(s))\|_{L_{2}(H_0,H)}\|h(s)\|_{H_0}ds\nonumber\\
&\leq& \frac{1}{2} \mathbb{E}\left[\sup_{s\in [0,t\wedge \tau_{N}]}\|u_{n}^{\epsilon,R}(s)\|^{p}\right] +K_1 p^2\mathbb{E} \int_{0}^{t\wedge \tau_{N}} \|h(s)\|_{H_0}^{2} \|u_{n}^{\epsilon,R}(s)\|^{p-2}(1+\|u_{n}^{\epsilon,R}(s)\|^2) ds\nonumber\\
&\leq&\frac{1}{2} \mathbb{E}\left[\sup_{s\in [0,t\wedge \tau_{N}]}\|u_{n}^{\epsilon,R}(s)\|^{p}\right] +C(K_1, p)\mathbb{E}\int_{0}^{t\wedge \tau_{N}} (1+ \|u_{n}^{\epsilon,R}(s)\|^{p}) \|h(s)\|_{H_0}^{2} ds,
\end{eqnarray*}
which by noting that $h\in \mathcal{S}_{M}$ and applying the Gronwall inequality, then letting $N$ tend to infinity lead to,
\begin{eqnarray}\label{3.3}
\sup_{\epsilon\in [0,1]}\mathbb{E}\left[\sup_{t\in[0, T]} \|u_{n}^{\epsilon,R}(t)\|^{p}\right] \leq C(M,T,p,K_1, \mathbb{E}\|u_0\|^p).
\end{eqnarray}
where $C$ is independent of $n, \epsilon, R$. Next, we apply the It\^{o} formula to $\|\partial_x u_{n}^{\epsilon,R}(t)\|^{p}$ to obtain,
\begin{eqnarray*}
d\|\partial_x u_{n}^{\epsilon,R}(t)\|^{p} &=& p \text{Im} \left(Au_{n}^{\epsilon,R}(t), Au_{n}^{\epsilon,R}(t)\right) \left\|\partial_x u_{n}^{\epsilon,R}(t)\right\|^{p-2}dt\\
&& -\lambda \text{Re}\left(\Psi(\|u^{\epsilon,R}_{n}(t)\|)f_{n}(u_{n}^{\epsilon,R}(t)), Au^{\epsilon,R}_{n}(t)\right) \left\|\partial_x u_{n}^{\epsilon,R}(t)\right\|^{p-2}dt \\
&& + p\sqrt{\epsilon}\text{Re} \left(g_{n}(t,u_{n}^{\epsilon,R}(t)), Au_{n}^{\epsilon,R}(t)\right)\left\|\partial_x u_{n}^{\epsilon,R}(t)\right\|^{p-2}dW\\
&&+p  \text{Re} \left(g_{n}(t,u_{n}^{\epsilon,R}(t))h(t), Au_{n}^{\epsilon,R}(t)\right) \left\|\partial_x u_{n}^{\epsilon,R}(t)\right\|^{p-2}dt\\
&&+ p\epsilon \|g_{n}(t,u_{n}^{\epsilon,R}(t))\|^{2}_{L_{2}(H_0, V)}\left\|\partial_x u_{n}^{\epsilon,R}(t)\right\|^{2p-2}dt\\
&&+ \epsilon \frac{p(p-1)}{2} \left|{\rm Re} \left(g_{n}(t, u_{n}^{\epsilon,R}(t)), Au_{n}^{\epsilon,R}(t)\right)\right|^{2} \left\|\partial_x u_{n}^{\epsilon,R}(t)\right\|^{p-4}dt
\end{eqnarray*}
where we again only determine an estimate for the fourth term and refer the reader to the proof of Theorem 4.2 of \cite{Lisei3}. We take the integral on time interval $[0,t\wedge \tau_{N}]$, where stopping time $\tau_{N}:= \inf\left\{t: \|\partial_{x}u_{n}^{\epsilon,R}(t)\|^{p}> N\right\}$, and $\tau_N=T$ when the set is empty. Then taking the expectation on both sides and estimating it using condition (\ref{2.7}) leads to,
\begin{eqnarray*}
&&\mathbb{E}\left|p \int_{0}^{t\wedge \tau_{N}} {\rm Re}\left(g_{n}(s,u_{n}^{\epsilon,R}(s))h(s), Au_{n}^{\epsilon,R}(s)\right) \left\|\partial_x u_{n}^{\epsilon,R}(s)\right\|^{p-2}ds\right|\\
&\leq& 2p \mathbb{E} \int_{0}^{t\wedge \tau_{N}} \left\|\partial_x u_{n}^{\epsilon,R}(s)\right\|^{p-2} \left\|\partial_x u_{n}^{\epsilon,R}(s)\right\| \left\|g_{n}(s,u_{n}^{\epsilon,R}(s))\right\|_{L_{2}(H_0,V)} \|h(s)\|_{H_0}ds\nonumber\\
&\leq&\frac{1}{2} \mathbb{E}\left[\sup_{s\in [0,t\wedge \tau_{N}]}\|\partial_xu_{n}^{\epsilon,R}(s)\|^{p}\right] +C(K_1, p)\mathbb{E}\int_{0}^{t\wedge \tau_{N}} (1+ \|\partial_xu_{n}^{\epsilon,R}(s)\|^{p}) \|h(s)\|_{H_0}^{2}) ds,\nonumber
\end{eqnarray*}
 and similar to the previous case, the Gronwall inequality gives
 \begin{eqnarray}\label{3.4}
\sup_{\epsilon\in [0,1]} \mathbb{E}\left[\sup_{s\in [0, t\wedge \tau_N]} \|u_{n}^{\epsilon,R}(t)\|_V^{p}\right] \leq C(M,T,p,K_1, \mathbb{E}\|u_0\|_V^p).
 \end{eqnarray}
Then, letting $N\rightarrow\infty$ yields the desired result.
 \end{proof}

With these estimates in hand, using the same argument as in \cite{Lisei3}, we may show that there exists a unique approximate solution $u_n^\epsilon$ to the finite dimensional system,
\begin{eqnarray}\label{3.5}
d(u^{\epsilon}_{n}(t), \varphi)&=& - i \langle Au^{\epsilon}_{n}(t), \varphi\rangle dt - \lambda \left(f_{n}(u^{\epsilon}_{n}(t)), \varphi\right)dt\nonumber \\
&& +\sqrt{\epsilon}\left(g_{n}(t, u^{\epsilon}_{n}(t))dW, \varphi\right)+\left(g_{n}(t,u^{\epsilon}_{n}(t))h, \varphi\right)dt,
\end{eqnarray}
for $\varphi\in H_n$. Moreover, we have the following a priori estimates
\begin{eqnarray}
&&\sup_{\epsilon\in [0,1]}\mathbb{E}\left[\sup_{ t\in [0, T]} \|u_{n}^{\epsilon}(t)\|^{p} \right]\leq C(M,T,p, K_1, \mathbb{E}\|u_0\|^p),\label{3.6}\\
&&\sup_{\epsilon\in [0,1]}\mathbb{E} \int_{0}^{T}  \|u_{n}^{\epsilon}(t)\|_V^{p}dt\leq C(M, T,p, K_1, \mathbb{E}\|u_0\|_V^p) \label{3.7}.
\end{eqnarray}
In addition, the following bounds may be achieved using the same estimates as lemma 3.1, after taking the inner product with $ u_{n,h}(t)$ and $\partial_{x} u_{n,h}(t)$, respectively,
\begin{eqnarray}\label{3.8}
\sup_{\epsilon\in [0,1]}\sup_{t\in [0,T]} \|u_{n,h}(t)\|^p\leq C(T,p), \text{   and  } ~\sup_{\epsilon\in [0,1]}\sup_{t\in [0,T]}\|\partial_{x} u_{n,h}(t)\|^p_{V} \leq C(T,p),\label{uM}
\end{eqnarray}
where $u_{n,h}$ is the approximate solution to the finite dimensional analogue of (\ref{controlled}).

Based on these estimates, we next prove the Theorem 1. Observe that by condition \eqref{2.6} and property (\ref{2.4}) we may obtain using (\ref{3.6}), (\ref{3.7}),
\begin{eqnarray}\label{3.9}
\mathbb{E}\int_{0}^{t} \|f(u_{n}^{\epsilon}(s))\|^{2}ds \leq 2^{2\sigma +1} \mathbb{E}\int_{0}^{t} \|u_{n}^{\epsilon}(s)\|_{V}^{4\sigma +2}ds\leq C(\sigma, M, p, T, K_1),
\end{eqnarray}
and
\begin{eqnarray}\label{3.10}
\mathbb{E}\int_{0}^{t} \|g(s,u_{n}^{\epsilon}(s))\|^{2}_{L_{2}(H_0,H)} ds \leq K_{1}\mathbb{E} \int_{0}^{t} \left(1+\|u_{n}^{\epsilon}(s)\|^{2}\right)ds \leq C(M, T, K_1),
\end{eqnarray}
where $C(M, T, K_1)$ is a constant independent of $n$. Then, the bounds (\ref{3.6})-(\ref{3.7}), (\ref{3.9})-(\ref{3.10}) and the Banach-Alaoglu theorem implies that there exist processes $u^\epsilon, \tilde{f}, \tilde{g}$ such that for any fixed $\epsilon >0$, as $n$ goes to $\infty$, we have, for $p\in [1,\infty)$,
\begin{eqnarray}
&&u_{n}^{\epsilon}\rightarrow u^{\epsilon}, \hspace{.4cm} \text{weak-* in } L^{p}(\Omega; \mathcal{C}([0,T]; H)), ~\text{weakly in } L^{p}(\Omega \times [0,T];V),\\
&&f(u_{n}^{\epsilon})\rightharpoonup \tilde{f}, \hspace{.4cm} \text{in } L^{p}(\Omega \times [0,T];H),\label{3.13}\\
&&g(s, u_{n}^{\epsilon}(s)) \rightharpoonup \tilde{g}, \hspace{.4cm} \text{in } L^{2}(\Omega \times [0,T]; L_{2}(H_0, H)),
\end{eqnarray}
where $\rightharpoonup$ denotes weak convergence. Using these weak limits and letting $n$ tend to infinity in \eqref{3.5}, we obtain that $u^\epsilon, \tilde{f}, \tilde{g}$ satisfy the following equation for all $t\in [0,T]$,
\begin{eqnarray}\label{3.14}
( u^{\epsilon}(t),\phi) &=& (u_{0},\phi) - i \int_{0}^{t} ( \nabla u^{\epsilon}(s),\nabla \phi) ds -\lambda \int_{0}^{t}\left( \tilde{f}(s),\phi\right) ds +\sqrt{\epsilon} \int_{0}^{t} \left( \tilde{g}(s),\phi\right) dW \nonumber \\
&&+ \int_{0}^{t}( \tilde{g}(s)h(s), \phi) ds,
\end{eqnarray}
for all $t\in [0,T]$ and $\phi\in V$. Next, we aim to identify the nonlinear term $\tilde{f}=f(u^\epsilon)$ and $\tilde{g}=g(s,u^\epsilon(s))$ $\mathbb{P}$ a.s. in (\ref{3.14}).  Let $\hat{u}_{n}^{\epsilon}:= \pi_{n}u^{\epsilon}$, $\hat{f}_{n}:= \pi_{n}\tilde{f}$ and $\hat{g}_{n}:= \pi_{n}\tilde{g}$. Then observe that $\hat{u}_{n}^{\epsilon}(t)$ satisfies the system,
\begin{eqnarray}
 (\hat{u}_{n}^{\epsilon}(t),\varphi)&=& (u_{0},\varphi) - i \int_{0}^{t}( A\hat{u}_{n}^{\epsilon}(s),\varphi) ds -\lambda \int_{0}^{t}\left( \tilde{f}(s),\varphi\right) ds + \sqrt{\epsilon}\int_{0}^{t} ( \tilde{g}(s),\varphi) dW \nonumber \\
&&+ \int_{0}^{t}(\tilde{g}(s)h(s), \varphi) ds,
\end{eqnarray}
for all $t\in [0,T]$, $\varphi \in H_n$. Taking the difference of $\hat{u}_n^\epsilon(t)$ with the solution $u_n^{\epsilon}(t)$ of equation \eqref{3.5}, then applying the It\^{o} formula to $\|u_{n}^{\epsilon}(t)-\hat{u}_n(t)\|^2$, we obtain,
\begin{eqnarray}\label{3.16}
\|u_{n}^{\epsilon}(t)-\hat{u}_{n}^{\epsilon}(t)\|^{2} &=& -2\text{Im}\int_{0}^{t}\left\langle A(u_{n}^{\epsilon}(s)-\hat{u}_{n}^{\epsilon}(s)), u_n^{\epsilon}(s)-\hat{u}_{n}^{\epsilon}(s)\right\rangle ds\nonumber\\
&&-2\lambda \text{Re} \int_{0}^{t} \left(f(u_{n}^{\epsilon}(s))-\tilde{f}(s), u_{n}^{\epsilon}(s)-\hat{u}_{n}^{\epsilon}(s)\right)ds\nonumber\\
&&+ 2\sqrt{\epsilon}\text{Re}\int_{0}^{t} \left(g(s,u_{n}^{\epsilon}(s))-\tilde{g}(s), u_{n}^{\epsilon}(s)-\hat{u}_{n}^{\epsilon}(s)\right)dW\nonumber \\
 &&+ 2\text{Re}\int_{0}^{t}\left(g(s,u_{n}^{\epsilon}(s))h(s)- \tilde{g}(s)h(s), u_{n}^{\epsilon}(s)-\hat{u}_{n}^{\epsilon}(s)\right)ds\nonumber\\
&&+\epsilon\int_{0}^{t} \|g(s, u_{n}^{\epsilon}(s))-\hat{g}_n(s)\|^{2}_{L_{2}(H_0,H)}ds.
\end{eqnarray}
By \eqref{2.3} and Young's inequality we may decompose the nonlinear term as follows,
\begin{eqnarray*}
&&-2\lambda \text{Re} \left(f(u_{n}^{\epsilon})-\tilde{f}(s), u_{n}^{\epsilon}-\hat{u}_{n}^{\epsilon}\right)\\
&=& -2\lambda \text{Re} \left(f(u_{n}^{\epsilon})-f(\hat{u}_{n}^{\epsilon}), u_{n}^{\epsilon}-\hat{u}_{n}^{\epsilon}\right) - 2\lambda \text{Re} \left(f(\hat{u}_{n}^{\epsilon})-f(u^{\epsilon}), u_{n}^{\epsilon}-\hat{u}_{n}^{\epsilon}\right)\\
&&-2\lambda \text{Re} \left(f(u^{\epsilon})-\tilde{f}(s), u_{n}^{\epsilon}-\hat{u}_{n}^{\epsilon}\right)\\
&\leq& 2\lambda^{2} \|f(\hat{u}_n^{\epsilon})-f(u^{\epsilon})\|^{2} + \frac{1}{2}\|u_{n}^{\epsilon}-\hat{u}_{n}^{\epsilon}\|^{2} - 2\lambda \text{Re} \left(f(u^{\epsilon})-\tilde{f}(s), u_{n}^{\epsilon}-\hat{u}_{n}^{\epsilon}\right).
\end{eqnarray*}
In addition, condition (\ref{2.8}) and the Cauchy-Schwarz inequality give,
\begin{eqnarray*}
&&\|g(s, u_{n}^{\epsilon}(s))-\hat{g}_n(s)\|^{2}_{L_{2}(H_0,H)}\\&\leq&\left\|g(s,u_{n}^{\epsilon})-\tilde{g}(s)\right\|^{2}_{L_{2}(H_0,H)}
+\|\hat{g}_n(s)-\tilde{g}(s)\|^{2}_{L_{2}(H_0,H)} \\
&\leq& \|g(s,u^{\epsilon})-g(s,u_{n}^{\epsilon})\|^{2}_{L_{2}(H_0,H)} -\|g(s,u^{\epsilon})-\tilde{g}(s)\|^{2}_{L_{2}(H_0,H)} \\
&&+ 2\left(\tilde{g}(s)-g(s,u_{n}^{\epsilon}), \tilde{g}(s)-g(s, u^{\epsilon})\right)+\|\hat{g}_n(s)-\tilde{g}(s)\|^{2}_{L_{2}(H_0,H)}\\
&\leq& K_{2}\|u_{n}^{\epsilon}-\hat{u}_{n}^{\epsilon}\|^{2} + K_{2}\|\hat{u}_{n}^{\epsilon}-u^{\epsilon}\|^{2} - \|g(s,u^{\epsilon})-\tilde{g}(s)\|^{2}_{L_{2}(H_0,H)} \\
&&+ 2\left(\tilde{g}(s)-g(s,u_{n}^{\epsilon}), \tilde{g}(s)-g(s,u^{\epsilon})\right)+\|\hat{g}_n(s)-\tilde{g}(s)\|^{2}_{L_{2}(H_0,H)},
\end{eqnarray*}
and
\begin{eqnarray*}
&&\text{Re}\left(g(s,u_{n}^{\epsilon}(s))h(s)-\tilde{g}(s)h(s), u_{n}^{\epsilon}(s)-\hat{u}_{n}^{\epsilon}(s)\right)\\
&=& \text{Re}\left(\left(g(s,u_{n}^{\epsilon}(s))h(s)-g(s,\hat{u}_{n}^{\epsilon}(s))\right)h(s), u_{n}^{\epsilon}(s)-\hat{u}_{n}^{\epsilon}(s)\right) \\
&&+ \text{Re}\left(\left(g(s,\hat{u}_{n}^{\epsilon}(s))-\tilde{g}(s)\right)h(s), u_{n}^{\epsilon}(s)-\hat{u}^{\epsilon}_{n}(s)\right)\\
&\leq& \frac{1}{2}\|u_{n}^{\epsilon}(s)-\hat{u}^{\epsilon}_{n}(s)\|^{2} + \frac{1}{2}K_{2} \|h(s)\|^{2}_{H_0} \|u^{\epsilon}_{n}(s)-\hat{u}_{n}^{\epsilon}(s)\|^{2}  \\
&&+ \text{Re}\left(\left(g(s,\hat{u}_{n}^{\epsilon}(s))-\tilde{g}(s)\right)h(s), u_{n}^{\epsilon}(s)-\hat{u}^{\epsilon}_{n}(s)\right).
\end{eqnarray*}
Now multiplying both sides of (\ref{3.16}) by $\phi(t)= \exp\left(-\int_{0}^{t} K_2\|h(s)\|^2_{H_0}+\epsilon K_2+1 ds\right)$ and then taking expectation yields,
\begin{eqnarray*}
&&\mathbb{E}\left[ \phi(t) \|u_{n}^{\epsilon}(t)-\hat{u}_{n}^{\epsilon}(t)\|^{2}\right] + \mathbb{E} \int_{0}^{t}\phi(s) \|u_{n}^{\epsilon}(s)-\hat{u}_{n}^{\epsilon}(s)\|^{2}(1+\epsilon K_2+ K_2\|h(s)\|^2_{H_{0}})ds\\
&&+ \epsilon\mathbb{E}\int_{0}^{t} \phi(s) \|g(s,u^{\epsilon}(s))-\tilde{g}(s)\|^{2}_{L_{2}(H_0,H)}ds\\
&\leq& C\lambda^2\mathbb{E} \int_{0}^{t} \phi(s) \left\|f(\hat{u}_{n}^{\epsilon}(s))-f(u^{\epsilon}(s))\right\|^{2} ds\\&&+ C\lambda\mathbb{E} \int_{0}^{t} \phi(s) \text{Re}\left(f(u^{\epsilon}(s))-\tilde{f}(s), u_{n}^{\epsilon}(s)-\hat{u}_{n}^{\epsilon}(s)\right)ds\\
&& + C\epsilon\mathbb{E} \int_{0}^{t} \phi(s) \left(\tilde{g}(s)-g(s,u_{n}^{\epsilon}(s)), \tilde{g}(s)-g_{n}(s, u^{\epsilon}(s))\right)ds \\
&&+ C \mathbb{E} \int_{0}^{t} \phi(s)\text{Re}\left(\left(\tilde{g}(s)-g(s,\hat{u}_{n}^{\epsilon}(s))\right)h(s), u_{n}^{\epsilon}(s)-\hat{u}_{n}^{\epsilon}(s)\right)ds\\
&&+C\epsilon\mathbb{E} \int_{0}^{t} \phi(s)\left(\|\hat{g}_n(s)-\tilde{g}(s)\|^{2}_{L_{2}(H_0,H)}+K_{2}\|\hat{u}_{n}^{\epsilon}-u^{\epsilon}\|^{2}\right)ds\\
&=& J_{1}+J_{2}+J_{3}+J_{4}+J_{5}.
\end{eqnarray*}
We proceed to estimate each term. Using \eqref{fsigma}, we have,
\begin{eqnarray*}\label{1.44}
J_{1}&\leq& 2^{2\sigma +1} (4\sigma -1)^{2} \mathbb{E} \int_{0}^{t} \phi(s) \left(\|\hat{u}_{n}^{\epsilon}(s)\|_{V}^{4\sigma} + \|u^{\epsilon}(s)\|_{V}^{4\sigma}\right) \|\hat{u}_{n}^{\epsilon}(s)-u^{\epsilon}(s)\|^{2}ds\nonumber\\
&\leq& C(R,\sigma)\left(\mathbb{E} \int_{0}^{t} \|\hat{u}_{n}^{\epsilon}(s)\|^{8\sigma}_{V} + \|u^{\epsilon}(s)\|_{V}^{8\sigma} ds\right)^{1/2}\left(\mathbb{E} \int_{0}^{t} \left\|\hat{u}_{n}^{\epsilon}(s)-u^{\epsilon}(s)\right\|^{4}ds\right)^{1/2}.
\end{eqnarray*}
 Since $\hat{u}_{n}^{\epsilon}=\pi_n u^{\epsilon}$, we have $\hat{u}_{n}^{\epsilon}(\cdot) \rightarrow u^{\epsilon}(\cdot)$ in $L^{4}(0,T;H)$ $\mathbb{P}$ a.s. for fixed $\epsilon>0$ as $n\rightarrow \infty$, combining this with the bound $\hat{u}_{n}^{\epsilon} \in  L^{p}(\Omega \times [0,T];V)$ for all $p\geq 1$, we apply the Vitali convergence theorem (Theorem \ref{6.2} in Appendix) to obtain,
\begin{eqnarray*}
\mathbb{E} \int_{0}^{T} \left\|\hat{u}_{n}^{\epsilon}(t)-u^{\epsilon}(t)\right\|^{4}dt \rightarrow 0, \hspace{.5cm} \text{as } \text{n} \rightarrow \infty.
\end{eqnarray*}
Therefore, noting estimate \eqref{3.7} we conclude that $J_{1} \rightarrow 0$ as $n\rightarrow \infty$. Moreover,
\begin{eqnarray*}
J_2&\leq& C\lambda\mathbb{E} \int_{0}^{t} \phi(s) \text{Re}\left(f(u^{\epsilon}(s))-\tilde{f}(s), u_{n}^{\epsilon}(s)-u^\epsilon(s)\right)ds\\
&&+C\left|\mathbb{E} \int_{0}^{t} \phi(s) \left(f(u^{\epsilon}(s))-\tilde{f}(s), u^{\epsilon}(s)-\hat{u}_{n}^{\epsilon}(s)\right)ds\right|.
\end{eqnarray*}
Now due to the facts that
\begin{eqnarray*}
u_{n}^\epsilon\rightarrow u^\epsilon,~ \text{weak-* in }~L^p(\Omega; \mathcal{C}([0,T];H)),
 \end{eqnarray*}
and
\begin{eqnarray*}
\hat{u}_{n}^\epsilon\rightarrow u^\epsilon,~{\rm in }~L^2(\Omega\times [0,T];H),
 \end{eqnarray*}
together with the bound $\tilde{f}\in L^{p}(\Omega \times [0,T];H)$, yield $|J_2|\rightarrow 0$ as $n\rightarrow\infty$. Similarly, as in estimate (\ref{3.10}), we have $g(s,u^{\epsilon}(s))\in L^{2}(\Omega \times [0,T];L_{2}(H_0, H))$, this along with the convergence $g(s, u_{n}^{\epsilon}(s))\rightharpoonup \tilde{g}$,  in $L^{2}(\Omega \times [0,T]; L_{2}(H_0,H))$ lead to
\begin{eqnarray*}
J_{3}= C\epsilon \mathbb{E}\int_{0}^{t} \phi(s) \left(\tilde{g}(s)-g(s, u_{n}^{\epsilon}(s)), \tilde{g}(s)-g(s, u^{\epsilon}(s))\right)ds \rightarrow 0,
\end{eqnarray*}
as $n\rightarrow \infty$. Furthermore,
\begin{eqnarray*}
J_{4}&=& \mathbb{E}\int_{0}^{t} \phi(s)\text{Re}\left((\tilde{g}(s)-g(s, u^{\epsilon}(s))+ g(s, u^{\epsilon}(s))-g(s, \hat{u}_{n}^{\epsilon}(s)))h(s), u_{n}^{\epsilon}(s)-\hat{u}_{n}^{\epsilon}(s)\right)ds\\
&\leq& \mathbb{E} \int_{0}^{t} \phi(s)\text{Re}\left((\tilde{g}(s)-g(s, u^{\epsilon}(s)))h(s), u_{n}^{\epsilon}(s)-\hat{u}_{n}^{\epsilon}(s)\right)ds\\
&& + \mathbb{E}\int_{0}^{t}\left\|g(s, u^{\epsilon}(s))-g(s, \hat{u}^{\epsilon}_{n}(s))\right\|_{L_{2}(H_0,H)}  \|u_n^{\epsilon}(s)-\hat{u}_{n}^{\epsilon}(s)\| \|h(s)\|_{H_0}ds\\
&=& \mathcal{J}_{1}+\mathcal{J}_{2}.
\end{eqnarray*}
Regarding $\mathcal{J}_{1}$, notice that $u_{n}^{\epsilon}(\cdot)-\hat{u}_{n}^{\epsilon}(\cdot) \rightharpoonup 0$ in $L^{2}(\Omega \times [0,T];H)$ for any fixed $\epsilon>0$ and we ensure that $(\tilde{g}-g(s, u^{\epsilon}(s)))h\in L^{2}(\Omega \times [0,T];H)$ by the following calculation,
\begin{eqnarray*}
&&\mathbb{E}\int_{0}^{T}\|(\tilde{g}(s)-g(s, u^{\epsilon}(s)))h(s)\|^{2}ds \\
&\leq& \mathbb{E}\int_{0}^{T} \left(\|\tilde{g}(s)\|^{2}_{L_{2}(H_0, H)} + \|g(s, u^{\epsilon}(s))\|^{2}_{L_{2}(H_0, H)}\right) \|h(s)\|^{2}_{H_0} ds\\
&\leq& \mathbb{E}\left[\sup_{s\in [0,T]} \left(\|\tilde{g}(s)\|^{2}_{L_{2}(H_0, H)} + \|g(s,u^{\epsilon}(s))\|^{2}_{L_{2}(H_0,H)}\right) \right]\int_{0}^{T} \|h(s)\|_{H_0}^{2}ds\\
&\leq& C(M,T, K_1),
\end{eqnarray*}
implying that $\mathcal{J}_{1} \rightarrow 0$ as $n\rightarrow \infty$. Now Young's inequality may be applied to $\mathcal{J}_{2}$ to obtain,
\begin{eqnarray*}
|\mathcal{J}_{2}| &\leq& K_2\mathbb{E} \left[\sup_{s\in[0,T]} \|u_n^{\epsilon}(s)-\hat{u}_{n}^{\epsilon}(s)\|\right]\int_{0}^{t} \|u^{\epsilon}(s)-\hat{u}_{n}^{\epsilon}(s)\| \|h(s)\|_{H_0}ds\\
&\leq& K_{2}\mathbb{E}\left[\sup_{s\in [0,T]}\|u_n^{\epsilon}(s)-\hat{u}_{n}^{\epsilon}(s)\|\right]\left(\int_0^t\|u^{\epsilon}(s)-\hat{u}_{n}^{\epsilon}(s)\|^2 ds\right)^\frac{1}{2}\left(\int_0^t\|h(s)\|_{H_0}^2ds\right)^\frac{1}{2} \\
&\leq& C(M,T,K_{2}) \left(\mathbb{E}\int_0^t\|u^{\epsilon}(s)-\hat{u}_{n}^{\epsilon}(s)\|^2ds\right)^\frac{1}{2}.
\end{eqnarray*}
Since $\hat{u}_{n}^{\epsilon}\rightarrow u^{\epsilon}$ in $L^2(\Omega\times [0,T]; H)$, we have $\mathcal{J}_2\rightarrow 0$ as $n\rightarrow\infty$. Finally, by the definition of $\hat{g}_n$ and $\hat{u}^\epsilon_n$, we have $J_5\rightarrow 0$ as $n\rightarrow\infty$. Thus,
\begin{eqnarray}\label{3.17}
\mathbb{E} \int_{0}^{t} \|u_{n}^{\epsilon}(s)-\hat{u}_{n}^{\epsilon}(s)\|^{2}ds\rightarrow 0,
\end{eqnarray}
as $n\rightarrow \infty$ and
\begin{eqnarray*}
\tilde{g}(s)=g(s, u^\epsilon(s)),~ \mathbb{P} ~\mbox {a.s.}
\end{eqnarray*}
We proceed to show that $\tilde{f}(s)=f(u^\epsilon)$, $\mathbb{P}$ a.s. Let $\varphi\in V$ and set $B\subset \Omega\times [0,T]$,
\begin{eqnarray*}
&&\mathbb{E}\int_{0}^{t}\left(\tilde{f}(s)-f(u^\epsilon), 1_{B}\varphi\right)ds
=\mathbb{E}\int_{0}^{t}\left(\tilde{f}(s)-f(u_n^\epsilon), 1_{B}\varphi\right)ds\\
&&+\mathbb{E}\int_{0}^{t}\left(f(u_n^\epsilon)-f(\hat{u}_n^\epsilon), 1_{B}\varphi\right)ds+\mathbb{E}\int_{0}^{t}\left(f(\hat{u}_n^\epsilon)-f(u^\epsilon), 1_{B}\varphi\right)ds\\ &&\leq \mathbb{E}\int_{0}^{t}\left(\tilde{f}(s)-f(u_n^\epsilon), 1_{B}\varphi\right)ds+\|\varphi\|\mathbb{E}\int_{0}^{t}\|f(\hat{u}_n^\epsilon)-f(u_n^\epsilon)\|
+\|f(\hat{u}_n^\epsilon)-f(u^\epsilon)\|ds\\ &&\leq C(\sigma,\|\varphi\|)\left(\mathbb{E} \int_{0}^{t} \|\hat{u}_{n}^{\epsilon}(s)\|^{4\sigma}_{V} + \|u^{\epsilon}(s)\|_{V}^{4\sigma}+\|u_n^{\epsilon}(s)\|_{V}^{4\sigma} ds\right)^{1/2}\nonumber\\ &&\times\left(\mathbb{E} \int_{0}^{t} \left\|\hat{u}_{n}^{\epsilon}(s)-u^{\epsilon}(s)\right\|^{2}
+\left\|\hat{u}_{n}^{\epsilon}(s)-u_n^{\epsilon}(s)\right\|^{2}ds\right)^{1/2}+\mathbb{E}\int_{0}^{t}\left(\tilde{f}(s)-f(u_n^\epsilon), 1_{B}\varphi\right)ds.
\end{eqnarray*}
By (\ref{3.13}), (\ref{3.17}) and the strong convergence $\hat{u}^\epsilon_n\rightarrow u^\epsilon$ in $L^2(\Omega\times[0,T];H)$, we have the right hand side going to $0$ as $n\rightarrow\infty$, which implies $\tilde{f}(s)=f(u^\epsilon)$, $\mathbb{P}$ a.s. using the density argument. Hence, we obtain the existence of solutions.

To show the uniqueness of solutions, let $v^{\epsilon}=u_{1}^{\epsilon}-u_{2}^{\epsilon}$ satisfying,
\begin{eqnarray*}
dv^\epsilon+ iAv^\epsilon dt &=& -\lambda(f(u_{1}^{\epsilon})-f(u_{2}^{\epsilon}))dt + \sqrt{\epsilon} (g(t, u_{1}^{\epsilon}(t))-g(t, u_{2}^{\epsilon}(t)))dW \\
&&+(g(t, u_{1}^{\epsilon}(t))-g(t, u_{2}^{\epsilon}(t)))h(t)dt,
\end{eqnarray*}
then applying the It\^{o} formula to $\|v^{\epsilon}(t)\|^{2}$ gives,
\begin{eqnarray*}
d\|v^{\epsilon}(t)\|^{2}&=& -2\lambda \text{Re}(f(u_{1}^{\epsilon}(s))-f(u_{2}^{\epsilon}(s)), v^{\epsilon}(s))ds \\
&&+ 2\sqrt{\epsilon} \text{Re} ( g(s, u_{1}^{\epsilon}(s))-g(s, u_{2}^{\epsilon}(s)),v^{\epsilon}(s))dW\\
&&+2\text{Re}([g(s, u_{1}^{\epsilon}(s))-g(s, u_{2}^{\epsilon}(s))]h(s), v^{\epsilon}(s)) ds \\
&&+ \epsilon\|g(s, u_{1}^{\epsilon}(s))-g(s, u_{2}^{\epsilon}(s))\|^{2}_{L_{2}(H_0,H)}ds.
\end{eqnarray*}
Observe that, by (\ref{2.8}),
\begin{eqnarray}\label{3.18}
&&\left|2\text{Re}([g(s, u_{1}^{\epsilon}(s))-g(s, u_{2}^{\epsilon}(s))]h(s), v^{\epsilon}(s))\right|\leq 2K_{2} \|v^{\epsilon}(s)\|^{2} \|h(s)\|_{H_0},
\end{eqnarray}
this together with (\ref{2.7}), property (\ref{2.3}) and (\ref{3.18}), we have,
\begin{eqnarray}\label{3.19}
d\|v^{\epsilon}(t)\|^{2}&\leq& 2K_{2} \|v^{\epsilon}(s)\|^{2} \|h(s)\|_{H_0}ds+\epsilon K_2\|v^{\epsilon}(s)\|^{2}ds\nonumber\\
&+ &2\sqrt{\epsilon} \text{Re} ( g(s, u_{1}^{\epsilon}(s))-g(s, u_{2}^{\epsilon}(s)),v^{\epsilon}(s))dW.
\end{eqnarray}
Now we multiply both sides of (\ref{3.19}) by $\phi(t)= \exp\left(-\int_{0}^{t} \epsilon K_2 + 2 K_2\|h(s)\|_{H_0}ds\right)$ and then take the expectation to obtain,
\begin{equation*}
\mathbb{E}\phi(s)\|v^{\epsilon}(s)\|^{2}\leq \mathbb{E}\int_0^t\phi(s)\|v^{\epsilon}(s)\|^{2}ds,
\end{equation*}
which implies that $u_{1}^{\epsilon}=u_{2}^{\epsilon}$ $\mathbb{P}$ a.s. for $t\in [0,T]$. This completes the proof of Theorem \ref{theorem1}.

Note that using estimates in \eqref{3.8} the existence and uniqueness of solutions to equation \eqref{controlled} may be obtained following the same lines of reasoning as above.

\section{Large Deviations}
\setcounter{equation}{0}
After achieving the well-posedness of solutions to equations, \eqref{original} and \eqref{controlled}, we now proceed to prove the large deviation principle by verifying the two conditions given in Theorem 6 in \cite{BDM} stated below. We begin by providing the definition of large deviation principle as follows for completeness.

\begin{definition} [Large Deviation Principle]
The sequence $\{X_{n}\}_{n\in \mathbb{N}}$ satisfies the LDP on polish space $\mathcal{E}$ with rate function $I$ if the following two conditions hold, \\
a. LDP lower-bound: for every open set $U\subset \mathcal{E}$,
\begin{eqnarray}\label{closed}
-\inf_{x\in U} I(x) \leq \liminf_{n\rightarrow \infty} \frac{1}{n} \log \mathbb{P}(X_{n} \in U),
\end{eqnarray}
b. LDP upper-bound: for every closed set $C\subset \mathcal{E}$,
\begin{eqnarray}\label{open}
\limsup_{n\rightarrow \infty} \frac{1}{n} \log \mathbb{P}(X_{n}\in C) \leq -\inf_{x\in C}I(x).
\end{eqnarray}
\end{definition}
More background on this area of study may be found in \cite{Dembo, Ellis}. To apply the weak convergence approach, we use the following theorem due to A. Budhiraja, P. Dupuis, V. Maroulas.

\begin{theorem}[Theorem 6 in \cite{BDM}]
Let $\mathcal{E}_{0}$ and $\mathcal{E}$ be two Polish spaces and family $\{X^{\epsilon,x}\}_{\epsilon>0}$ be defined by $X^{\epsilon,x}:= \mathcal{G}(x, \sqrt{\epsilon}W)$ for a measurable map $\mathcal{G}^{\epsilon}:\mathcal{E}_{0} \times \mathcal{C}([0,T];H)\rightarrow \mathcal{E}$. If there exists a measurable map $\mathcal{G}^{0}:\mathcal{E}_{0} \times \mathcal{C}([0,T];H) \rightarrow \mathcal{E}$ such that, \\
$i$. For any fixed $M<\infty$ and compact set $K\subset \mathcal{E}_{0}$,
\begin{eqnarray}
K_{M}= \left\{\mathcal{G}^{0}\left(x,\int_{0}^{\cdot}h(s)ds\right), h\in \mathcal{S}_{M}(H_0), x\in K\right\},
\end{eqnarray}
is a compact subset of $\mathcal{E}$. \\
$ii.$ For $M<\infty$ and $\{h_{\epsilon}\}_{\epsilon>0} \subset \mathcal{S}_{M}(H_0)$ and $\{x^{\epsilon}\}_{\epsilon>0}\subset \mathcal{E}_{0}$ if as $\epsilon$ tends to zero, $h_{\epsilon}$ converges in distribution to $h$ and $x^{\epsilon}$ converges to $x$ then,
\begin{eqnarray}
\mathcal{G}^{\epsilon}\left(x^{\epsilon}, \sqrt{\epsilon}W+ \int_{0}^{\cdot} h_{\epsilon}(s)ds\right)\rightarrow \mathcal{G}^{0}\left(x, \int_{0}^{\cdot}h(s)ds\right),
\end{eqnarray}
in distribution as $\epsilon$ tends to zero.

Then $\{X^{\epsilon, x}\}_{\epsilon>0}$ satisfies the large deviation principle with rate function,
\begin{eqnarray}
I_{x}(f)= \inf_{\left\{h\in L^{2}(0,T;H_0): f= \mathcal{G}^{0}\left(x, \int_{0}^{\cdot}h(s)ds\right)\right\}} \frac{1}{2} \int_{0}^{T}\|h(s)\|_{H_0}^{2}ds,
\end{eqnarray}
where, $x\in \mathcal{E}_{0}$, $f\in \mathcal{E}$.
\end{theorem}
In the stochastic PDE setting, the maps $\mathcal{G}^{\epsilon}, \mathcal{G}^{0}$ correspond to the unique variational solution to \eqref{original} and (\ref{controlled}), respectively. Thanks to the uniqueness of solutions, $\mathcal{G}^{\epsilon}(u_0^{\epsilon}, \sqrt{\epsilon}W + \int_{0}^{\cdot}h_{\epsilon}(s)ds)$ is the solution of the following stochastic control equation,
\begin{eqnarray}\label{4.6}
\left\{ \begin{array}{ll}
du_{h_\epsilon}^{\epsilon}(t)= - iAu_{h_\epsilon}^{\epsilon}(t)dt - \lambda f(u^{\epsilon}_{h_{\epsilon}}(t))dt +  g(t, u_{h_{\epsilon}}^{\epsilon}(t))h_{\epsilon}(t)dt +\sqrt{\epsilon} g(t, u_{h_{\epsilon}}^{\epsilon}(t))dW,\\
u^{\epsilon}_{h_{\epsilon}}(0)= u_0^\epsilon.
\end{array} \right.
\end{eqnarray}

In order to verify the two conditions needed by Theorem 4.1, we first recall the fractional order Sobolev space $W^{\alpha, p}$, which is required to establish the tightness of the probability measures. It is natural to consider these spaces since the solution of the stochastic evolution system are H\"{o}lder continuous of order strictly less than $\frac{1}{2}$ with respect to time.
Let
\begin{eqnarray*}
W^{\alpha,p}(0,T;X):=\left\{v\in L^{p}(0,T;X):\int_{0}^{T}\int_{0}^{T}\frac{\|v(t_{1})-v(t_{2})\|_{X}^{p}}{|t_{1}-t_{2}|^{1+\alpha p}}dt_{1}dt_{2}<\infty\right\},
\end{eqnarray*}
for any fixed $p>1$ and $\alpha\in(0,1)$, endowed with the norm,
\begin{eqnarray*}
\|v\|^p_{W^{\alpha,p}(0,T;X)}:=\int_{0}^{T}\|v(t)\|_{X}^{p}dt+\int_{0}^{T}\int_{0}^{T}\frac{\|v(t_{1})-v(t_{2})\|_{X}^{p}}{|t_{1}-t_{2}|^{1+\alpha p}}dt_{1}dt_{2},
\end{eqnarray*}
where $X$ is a separable Hilbert space.
For the case $\alpha=1$, we take,
\begin{eqnarray*}
W^{1,p}(0,T;X):=\left\{v\in L^{p}(0,T;X):\frac{dv}{dt}\in L^{p}(0,T;X)\right\},
\end{eqnarray*}
which is the classical Sobolev space with its usual norm,
\begin{eqnarray*}
\|v\|_{W^{1,p}(0,T;X)}^{p}:=\int_{0}^{T}\|v(t)\|_{X}^{p}+\left\|\frac{dv}{dt}(t)\right\|_{X}^{p}dt.
\end{eqnarray*}
Note that for $\alpha\in(0,1)$, the embedding of $ W^{1,p}(0,T;X)$ into $W^{\alpha,p}(0,T;X)$ also holds.

\begin{lemma}\label{lem4.1} For $\{h_{\epsilon}\}_{\epsilon >0} \subset \mathcal{S}_{M}(H_{0})$, the following convergence holds  $\mathbb{P}$ \mbox{a.s.}
\begin{eqnarray}
u^\epsilon_{h_{\epsilon}}\rightarrow u_h, ~{\rm in}~ L^2(0,T;H),
\end{eqnarray}
as $\epsilon\rightarrow 0$, where $u^\epsilon_{h_\varepsilon},
u_h$ are solutions corresponding to systems (\ref{4.6}), and equation (\ref{controlled}), respectively.
\end{lemma}
\begin{proof} We first show that
\begin{eqnarray}
&&\mathbb{E}\left\|u_{h_{\epsilon}}^\epsilon-\sqrt{\epsilon} \int_{0}^{t} g(s, u_{h_{\epsilon}}^{\epsilon}(s))dW\right\|_{W^{1,2}(0,T;V')}^2\leq C,\label{4.8}\\
&&\mathbb{E}\left\|\sqrt{\epsilon} \int_{0}^{t} g(s, u_{h_{\epsilon}}^{\epsilon}(s))dW\right\|_{W^{\alpha,p}(0,T;V')}^p\leq C, \label{4.9}
\end{eqnarray}
for $\alpha\in [0,\frac{1}{2})$ and $p\in [1,\infty)$, where $C$ is a constant independent of $\epsilon$. We proceed as follows,
\begin{eqnarray*}
u_{h_{\epsilon}}^{\epsilon}(t) &=& u_{0}^{\epsilon} - i\int_{0}^{t} Au_{h_{\epsilon}}^{\epsilon}(s)ds - \lambda \int_{0}^{t} f(u_{h_{\epsilon}}^{\epsilon}(s))ds + \int_{0}^{t} g(s, u_{h_{\epsilon}}^{\epsilon}(s))h_{\epsilon}(s)ds \\ &&+\sqrt{\epsilon} \int_{0}^{t} g(s, u_{h_{\epsilon}}^{\epsilon}(s))dW\\
&=& I_{0}+I_{1}+I_{2}+I_{3}+I_{4}.
\end{eqnarray*}
Note that by definition,
\begin{eqnarray*}
\mathbb{E}\|I_{1}\|^{2}_{W^{1,2}(0,T;V')}&=&\mathbb{E} \int_{0}^{T} \|Au_{h_{\epsilon}}^{\epsilon}(s)\|_{V'}^{2}ds+\mathbb{E} \int_{0}^{T}\left\|i\int_{0}^{t} Au_{h_{\epsilon}}^{\epsilon}(s)ds \right\|_{V'}dt\\
&\leq& C(T)\mathbb{E} \int_{0}^{T} \|Au_{h_{\epsilon}}^{\epsilon}(s)\|_{V'}^{2}ds\\ &\leq& C(T)\mathbb{E} \int_{0}^{T} \|u^{\epsilon}_{h_{\epsilon}}(s)\|_{V}^{p}ds \leq C(M,p,T,K_1).
\end{eqnarray*}
 Furthermore, the nonlinear term can be controlled using (\ref{2.4}) as follows,
\begin{eqnarray*}
\lambda \mathbb{E} \int_{0}^{T} \left\|\int_{0}^{t} |u_{h_{\epsilon}}^{\epsilon}(s)|^{2\sigma} u_{h_{\epsilon}}^{\epsilon}(s)ds\right\|_{V'}^{2}dt&\leq& \lambda T \mathbb{E} \int_{0}^{T}\||u_{h_{\epsilon}}^{\epsilon}(t)|^{2\sigma}u_{h_{\epsilon}}^{\epsilon}(t)\|_{V'}^{2}dt\\ &\leq& \lambda T \mathbb{E} \int_{0}^{T}\||u_{h_{\epsilon}}^{\epsilon}(t)|^{2\sigma}u_{h_{\epsilon}}^{\epsilon}(t)\|^{2}dt\\&\leq& C(\lambda, T, M, \sigma)\mathbb{E} \int_{0}^{T}\|u_{h_{\epsilon}}^{\epsilon}(t)\|_{V}^{4\sigma+2}dt\\ &\leq& C(\lambda, M,T, p, K_1, \sigma).
\end{eqnarray*}
Moreover, by (\ref{2.6})
\begin{eqnarray*}
\mathbb{E} \int_{0}^{T} \|I_{3}\|_{V'}^{2}dt &\leq& \mathbb{E} \int_{0}^{T} \left|\int_{0}^{t} \|g(s, u_{h_{\epsilon}}^{\epsilon}(s))\|_{L_{2}(H_0,H)} \|h_{\epsilon}(s)\|_{H_0} ds \right|^{2}dt \\
&\leq& C\mathbb{E} \int_{0}^{T} \int_{0}^{t}K_{1}(1+\|u_{h_{\epsilon}}^{\epsilon}(s)\|^{2})ds \int_{0}^{t} \|h_{\epsilon}(s)\|_{H_0}^{2}dsdt\\
&\leq& CMTK_{1}\mathbb{E} \int_{0}^{T} (1+ \|u_{h_{\epsilon}}^{\epsilon}(s)\|^{2})ds\leq C(M,T,K_1).
\end{eqnarray*}
Thus, we obtain the bound (\ref{4.8}). For bound (\ref{4.9}), the Burkholder-Davis-Gundy inequality (\ref{2.2}) may be applied as follows,
\begin{eqnarray*}
&&\mathbb{E}\left\|\sqrt{\epsilon} \int_{0}^{t} g(s, u_{h_{\epsilon}}^{\epsilon}(s))dW\right\|_{W^{\alpha,p}(0,T;V')}^p
\leq
\mathbb{E}\int_{0}^{T}\int_{0}^{T}\frac{\left\| \sqrt{\epsilon} \int_{s}^{t} g(r, u_{h_{\epsilon}}^{\epsilon}(r))dW\right\|^p_{V'}}{|t-s|^{1+\alpha p}}dsdt\\
&\leq& \int_{0}^{T}\int_{0}^{T}\frac{\epsilon^\frac{p}{2}\mathbb{E}\left[\int_{s}^{t}\|g(r, u_{h_{\epsilon}}^{\epsilon}(r))\|_{V'}^2dr\right]^{\frac{p}{2}}}{|t-s|^{1+\alpha p}}dsdt\\
&\leq& \epsilon^\frac{p}{2}K_1^\frac{p}{2}\int_{0}^{T}\int_{0}^{T}\frac{(t-s)^\frac{p}{2}\cdot\mathbb{E}\left[1+\sup_{r\in [0,T]}\|u_{h_{\epsilon}}^{\epsilon}(r)\|^p\right]}{|t-s|^{1+\alpha p}}dsdt\\
&\leq& C(\epsilon, T, p, K_1)\int_{0}^{T}\int_{0}^{T}|t-s|^{1+\left(\alpha-\frac{1}{2}\right)p}dtds\leq C( T, p, K_1).
\end{eqnarray*}

Define the set of probability measures,
\begin{eqnarray}\nu^\epsilon(\mathcal{O})=\mathbb{P}(u^\epsilon_{h_{\epsilon}}\in \mathcal{O}),
\end{eqnarray}
for any set $\mathcal{O}\in \mathcal{B}(X)$, where $X$ is the path space $L^2(0,T; H)$.
Based on the bounds (\ref{4.8}), (\ref{4.9}), and the fact that $L^{2}(0,T;V) \cap W^{\alpha,p}(0,T;V')$ is compactly embedded in $L^{2}(0,T;H)$ for any $p\in (1,\infty)$, $\alpha\in (0,1)$, we have that the set of probability measures $\{\nu^\epsilon\}_{\epsilon>0}$ is tight on the path space $L^2(0,T; H)$.
Indeed, note that
\begin{eqnarray*}
B_{K}:=\left\{u^\epsilon_{h_{\epsilon}}\in L^{2}(0,T; V)\cap W^{\alpha,p}(0,T;V'): \|u^\epsilon_{h_{\epsilon}}\|_{L^{2}(0,T;V)}+\|u^\epsilon_{h_{\epsilon}}\|_{W^{\alpha,p}(0,T;V')}\leq K\right\},
\end{eqnarray*}
 is compact on $L^{2}(0,T; H)$. Moreover, the bounds (\ref{4.8}), (\ref{4.9}), along with the Chebyshev inequality give,
\begin{eqnarray*}
\nu^{\epsilon}\left(B_{K}^c\right)&=&\mathbb{P}(\|u^\epsilon_{h_{\epsilon}}\|_{L^{2}(0,T;V)}+\|u^\epsilon_{h_{\epsilon}}\|_{W^{\alpha,p}(0,T;V')}> K)\\ &\leq& \frac{1}{K}\mathbb{E}(\|u^\epsilon_{h_{\epsilon}}\|_{L^{2}(0,T;V)}+\|u^\epsilon_{h_{\epsilon}}\|_{W^{\alpha,p}(0,T;V')})\leq \frac{C}{K}.
\end{eqnarray*}

Using the Skorokhod representation theorem (see Appendix), we then deduce the existence of a stochastic basis $(\widetilde{\Omega}, \widetilde{\mathcal{F}}, \{\widetilde{\mathcal{F}}_{t}\}_{t}, \widetilde{\mathbb{P}})$ in which there are processes $\{\tilde{u}^{\epsilon}_{h_{\epsilon}}\}_{\epsilon>0}$, such that $\tilde{u}^{\epsilon}_{h_{\epsilon}}$ has the same distribution as $u^{\epsilon}_{h_{\epsilon}}$ and $\widetilde{\mathbb{P}}$ a.s. $\tilde{u}^{\epsilon}_{h_{\epsilon}}$ converges to $\tilde{u}_{h}$ in $L^{2}(0,T;H)$ as $\epsilon\rightarrow 0$. With the uniqueness of solutions, we recover the convergence on the original probability space $(\Omega, \mathcal{F},\{\mathcal{F}_t\}_{t>0}, \mathbb{P})$ using Gy\"{o}ngy-Krylov's lemma, for more details, see \cite{ours}, that is,
\begin{eqnarray}\label{con}
u^{\epsilon}_{h_{\epsilon}}\rightarrow u_h, ~{\rm  in} ~L^{2}(0,T;H), ~\mathbb{P}~ \mbox{a.s.}
\end{eqnarray}
  In the spirit of \cite{Bessaih}, we now show that $u_h$ is the solution to equation (2.11). For any $\phi\in H^1$, and set $B\subset \Omega\times [0,T]$,
\begin{eqnarray}\label{4.11}
&&(u_{h_{\epsilon}}^\epsilon(t), 1_B \phi)-\left(i\int_{0}^{t} \langle Au_h, 1_B \phi\rangle ds - \lambda \int_{0}^{t} (f(u_h),1_B\phi)ds + \int_{0}^{t}( g(s, u_h)h,1_B\phi)ds\right)\nonumber\\
&=&i\int_{0}^{t} (\nabla u_{h_{\epsilon}}^\epsilon-\nabla u_h, 1_B \nabla\phi)ds-\lambda \int_{0}^{t} (f(u^\epsilon_{h_{\epsilon}})-f(u_h),1_B\phi)ds\nonumber\\&&\quad+\sqrt{\epsilon}\int_{0}^{t} ( g(s, u_{h_{\epsilon}}^{\epsilon}(s)), 1_B\phi) dW
+\int_{0}^{t}( g(s,u^\epsilon_{h_{\epsilon}})h_{\epsilon}-g(s, u_h)h,1_B\phi)ds.
\end{eqnarray}
We now show that the expectation of the right hand side of (\ref{4.11}) tends to $0$ as $\epsilon\rightarrow 0$. Since $u_{h_{\epsilon}}^\epsilon\rightarrow u_h$ weakly in $L^{p}(\Omega\times [0,T]; V)$ for $p\geq 1$ we have,
\begin{eqnarray}
\mathbb{E}\int_{0}^{t} i(\nabla u_{h_{\epsilon}}^\epsilon-\nabla u_h, 1_B \nabla\phi)ds\rightarrow 0,~ {\rm as}~ \epsilon\rightarrow 0.
\end{eqnarray}
Inequalities (\ref{2.5}), (\ref{2.13}) and the Cauchy-Schwarz inequality yield,
\begin{eqnarray}
&&\lambda\mathbb{E}\int_{0}^{t}(f(u^\epsilon_{h_{\epsilon}})-f(u_h),1_B\phi)ds \leq \lambda\|\phi\|\mathbb{E}\int_{0}^{t}\|f(u^\epsilon_{h_{\epsilon}})-f(u_h)\|ds\nonumber\\
&\leq& C(\lambda, \|\phi\|, \alpha)\mathbb{E}\int_{0}^{t}(\|u^\epsilon_{h_{\epsilon}}\|_{V}^{2\sigma}
+\|u_{h_{\epsilon}}\|_{V}^{2\sigma})\|u^\epsilon_{h_{\epsilon}}-u_h\|ds\nonumber\\
&\leq& C(\lambda, \|\phi\|, \alpha)\left(\mathbb{E}\int_{0}^{t}\|u^\epsilon_{h_{\epsilon}}\|_{V}^{4\sigma}
+\|u_{h}\|_{V}^{4\sigma}ds\right)^\frac{1}{2}\left(\mathbb{E}\int_{0}^{t}\|u^\epsilon_{h_{\epsilon}}-u_h\|^2ds\right)^\frac{1}{2}\nonumber\\
&\leq& C(\lambda,\|\phi\|, T, p, \alpha)\left(\mathbb{E}\int_{0}^{t}\|u^\epsilon_{h_{\epsilon}}-u_h\|^2ds\right)^\frac{1}{2}\rightarrow 0,
\end{eqnarray}
where the last step is deduced from the result of (\ref{con}) along with the Vitali convergence theorem. Using the Burkholder-Davis-Gundy inequality (\ref{2.2}) and condition (\ref{2.6}), we have as $\epsilon\rightarrow 0$,
\begin{eqnarray}
\mathbb{E}\left[\sqrt{\epsilon}\int_{0}^{t} ( g(s, u_{h_{\epsilon}}^{\epsilon}(s)), 1_B \phi ) dW\right]
&\leq& \sqrt{\epsilon}\mathbb{E}\left(\int_0^t( g(s, u_{h_{\epsilon}}^{\epsilon}(s)), 1_B \phi )^2ds\right)^\frac{1}{2}\nonumber\\
&\leq &\|\phi\|\sqrt{K_1 T\epsilon}\mathbb{E}\left[\sup_{t\in[0,T]}(1+\|u_{h_{\epsilon}}^{\epsilon}(t)\|)\right]\nonumber\\&\leq& C(\|\phi\|, T, K_1)\sqrt{\epsilon}\rightarrow 0.
\end{eqnarray}
With regards to the last term in \eqref{4.11}, we use the decomposition,
\begin{eqnarray*}
&&\quad\mathbb{E}\int_{0}^{t}( g(s,u^\epsilon_{h_{\epsilon}})h_{\epsilon}-g(s, u_h)h,1_B\phi)ds\nonumber\\
&&=\mathbb{E}\int_{0}^{t}((g(s,u^\epsilon_{h_{\epsilon}})-g(s, u_h))h_{\epsilon},1_B\phi)ds+\mathbb{E}\int_{0}^{t}(g(s, u_h)(h_{\epsilon}-h),1_B\phi)ds\nonumber \\&&=J_1+J_2.
\end{eqnarray*}
Applying condition (\ref{2.8}) and the H\"{o}lder inequality, $J_1$ can be bounded as follows,
\begin{eqnarray}\label{4.16}
|J_1|&\leq& K_2\|\phi\| \mathbb{E}\int_{0}^{t}\|u^\epsilon_{h_{\epsilon}}-u_h\|\|h_{\epsilon}\|_{H_0}ds\nonumber\\ &\leq& C( K_2, \|\phi\|,M)\left(\mathbb{E}\int_{0}^{t}\|u^\epsilon_{h_{\epsilon}}-u_h\|^2ds\right)^\frac{1}{2}\rightarrow 0.
\end{eqnarray}
Finally, by condition (\ref{2.6}) and bound in (\ref{3.8}), we have,
\begin{eqnarray}\label{4.17}
\int_{0}^{t}\|g^*(s, u_h)1_B \phi \|^2_{H_0}ds\leq K_1\|\phi\|^2\int_0^t(1+\|u_h\|^2)ds\leq C(K_1, \|\phi\|, T).
\end{eqnarray}
Now the set $\mathcal{S}_{M}$ being closed and bounded implies that $h_{\epsilon}\rightarrow h$ weakly in $L^{2}(0,T;H_0)$ $\mathbb{P}$ a.s. and with $(\ref{4.17})$ we have that $\mathbb{P}$ a.s.
\begin{eqnarray}\label{4.18}
\int_{0}^{t}(h_{\epsilon}-h,g^*(s, u_h)1_B\phi)ds\rightarrow 0.
\end{eqnarray}
Furthermore, (\ref{4.18}) and the Vitali convergence theorem give $J_2\rightarrow 0$ as $\epsilon\rightarrow 0$. Combining these estimates, we arrive at,
\begin{eqnarray}\label{4.19}
&&\mathbb{E}\left[(u_{h_{\epsilon}}^\epsilon, 1_B \phi)-\left(i\int_{0}^{t} \langle Au_h, 1_B \phi\rangle ds - \lambda \int_{0}^{t} (f(u_h),1_B\phi)ds +\int_{0}^{t}( g(s, u_h)h,1_B\phi)ds\right)\right]\nonumber \\ &&\rightarrow 0, ~~ {\rm  as} ~~\epsilon\rightarrow 0.
\end{eqnarray}
On the other hand, since $u_{h_{\epsilon}}^\epsilon\in L^p(\Omega;\mathcal{C}([0,T];H))$ uniformly in $\epsilon$, we have by the Banach-Alaoglu theorem, as $\epsilon\rightarrow 0$,
\begin{eqnarray}\label{4.20}
\mathbb{E}\left[\sup_{t\in [0,T]}\left(u_{h_{\epsilon}}^\epsilon(t)-u_h(t),1_B\phi\right)\right]\rightarrow 0,
\end{eqnarray}
  and thus we infer from (\ref{4.19}) and (\ref{4.20}) that $u_h$ is a solution of equation (\ref{controlled}) completing the proof.
\end{proof}

For $p\geq 1$, let the polish space
\begin{eqnarray}
\mathcal{X}:=\mathcal{C}([0,T]; H)\cap L^p(0,T;V),
\end{eqnarray}
be endowed with the norm $\|u\|^2_{\mathcal{X}}=\sup_{t\in [0,T]}\|u(t)\|^2$.
\begin{proposition}\label{pro4.1}
 For a positive constant $M$, if $\{h_{\epsilon}\}_{\epsilon>0}$ is a sequence in $\mathcal{S}_{M}(H_0)$ such that $h_{\epsilon}$ converges in distribution to $h$ as $\epsilon \rightarrow 0$ then,
 \begin{eqnarray*}
 \mathcal{G}^{\epsilon}\left(\sqrt{\epsilon}W + \int_{0}^{\cdot}h_{\epsilon}(s)ds\right) \rightarrow \mathcal{G}^{0}\left(\int_{0}^{\cdot}h(s)ds\right),
 \end{eqnarray*}
 in distribution in $\mathcal{X} $as $\epsilon \rightarrow 0$, where $\mathcal{G}^{\epsilon}(\sqrt{\epsilon}W)$ denotes the unique variational solution of equation \eqref{2.1}.
\end{proposition}

\begin{proof}
 Denote $\mathcal{G}^{\epsilon}\left(\sqrt{\epsilon} W + \int_{0}^{\cdot} h_{\epsilon}(s)ds\right)$ as the variational solution of \eqref{4.6}. Due to the uniqueness of solutions, we have $\mathcal{G}^{\epsilon}\left(\sqrt{\epsilon} W + \int_{0}^{\cdot} h_{\epsilon}(s)ds\right)=u_{h_{\epsilon}}^\epsilon(\cdot)$. Let $V^{\epsilon}(t):= u_{h_{\epsilon}}^{\epsilon}(t)-u_{h}(t)$, then $V^\epsilon(t)$ satisfies
\begin{eqnarray*}
V^{\epsilon}(t) &=& -i\int_{0}^{t} AV^{\epsilon}(s)ds - \lambda \int_{0}^{t} f(u^{\epsilon}_{h_{\epsilon}}(s))-f(u_{h}(s))ds+ \sqrt{\epsilon} \int_{0}^{t} g(s, u_{h_{\epsilon}}^{\epsilon}(s))dW \\
&&+ \int_{0}^{t} g(s ,u_{h_{\epsilon}}^{\epsilon}(s))(h_{\epsilon}(s)-h(s))ds+ \int_{0}^{t}(g(s, u_{h_{\epsilon}}^{\epsilon}(s))-g(s, u_{h}(s)))h(s)ds .
\end{eqnarray*}
For $N\in \mathbb{N}$, let $\tau_{N}:= \inf\left\{t: \|V^{\epsilon}(t)\|^{2}>N\right\}$, and $\tau_{N}=T$ if the set is empty. Applying the It\^{o} formula to $\|V^\epsilon\|^2$, then taking the supremum over interval $[0, t\wedge \tau_{N}]$, and afterwards the expectation we obtain,
\begin{eqnarray*}
\mathbb{E}\left[\sup_{s\in[0, t\wedge \tau_{N}]} \|V^{\epsilon}(s)\|^{2}\right] &\leq& 2\mathbb{E}\text{Im} \int_{0}^{t\wedge \tau_{N}} \|\nabla V^{\epsilon}(s)\|^{2}ds \\
&&- 2\lambda \mathbb{E}\text{Re} \int_{0}^{t\wedge \tau_{N}} \left(f(u^{\epsilon}_{h_{\epsilon}}(s))-f(u_{h}(s)), V^{\epsilon}(s)\right)ds\\
&& + 2 \mathbb{E} \left|{\rm Re} \int_{0}^{t\wedge \tau_{N}} \left(g(s, u_{h_{\epsilon}}^{\epsilon}(s))\left(h_{\epsilon}(s)-h(s)\right), V^{\epsilon}(s)\right)ds\right| \\
&&+ 2\mathbb{E} \left|{\rm Re}\int_{0}^{t\wedge \tau_{N}}\left(\left(g(s, u^{\epsilon}_{h_{\epsilon}}(s))-g(s, u_{h}(s))\right)h(s),V^{\epsilon}(s)\right)ds\right|\\
&&+ 2 \sqrt{\epsilon}\mathbb{E}\left[\sup_{s\in [0, t\wedge \tau_{N}]} \left|{\rm Re}\int_{0}^{s} \left(g(r, u_{h_{\epsilon}}^{\epsilon}(r))dW, V^{\epsilon}(r)\right)\right|\right]\\&&+ \epsilon \mathbb{E} \int_{0}^{t\wedge \tau_{N}} \|g(s, u_{h_{\epsilon}}^{\epsilon}(s))\|_{L_{2}(H_0,H)}^{2}ds.
\end{eqnarray*}
By the properties of $f(\cdot)$, conditions on $g(\cdot)$, the Cauchy-Schwartz inequality and Burkholder-Davis-Gundy inequality (\ref{2.2}), we arrive at,

\begin{eqnarray*}
&&\mathbb{E}\left[\sup_{s\in [0, t\wedge \tau_{N}]}\|V^{\epsilon}(s)\|^{2}\right]\\ &\leq& 2 \mathbb{E} \int_{0}^{t \wedge \tau_{N}} \|g(s, u_{h_{\epsilon}}^{\epsilon}(s))\|_{L_{2}(H_0,H)} \|h_{\epsilon}(s)-h(s)\|_{H_0} \|V^{\epsilon}(s)\|ds\\
&&+ 2\mathbb{E} \int_{0}^{t\wedge \tau_{N}} \|h(s)\|_{H_0} \|g(s, u_{h_{\epsilon}}^{\epsilon}(s))-g(s, u_{h}(s))\|_{L_{2}(H_0,H)} \|V^{\epsilon}(s)\|ds\\
&& + 2\sqrt{\epsilon} \mathbb{E} \left(\int_{0}^{t\wedge \tau_{N}} \|g(s, u^{\epsilon}_{h_{\epsilon}}(s))\|_{L_{2}(H_0,H)}^{2} \|V^{\epsilon}(s)\|^{2}ds \right)^{1/2}\\
&& + \epsilon \mathbb{E} \int_{0}^{t\wedge \tau_{N}} K_{1}(1+ \|u_{h_{\epsilon}}^{\epsilon}(s)\|^{2})ds\\
&=& I_{1} + I_{2}+I_{3}+I_{4}.
\end{eqnarray*}
Next, we show that the terms on the right hand side converge to zero as $\epsilon\rightarrow 0$. For $I_1$,
\begin{eqnarray*}
|I_{1}| &\leq& 2K_1\mathbb{E} \int_{0}^{t \wedge \tau_{N}}(1+\|u^\epsilon_{h_{\epsilon}}\|)  \|h_{\epsilon}(s)-h(s)\|_{H_0} \|V^{\epsilon}(s)\|ds\\&\leq&2 K_1\mathbb{E}\left[\sup_{t\in[0,T]}(1+\|u^\epsilon_{h_{\epsilon}}\|)\left(\int_{0}^{t\wedge \tau_{N}}\|h_{\epsilon}(s)-h(s)\|_{H_{0}}^{2}ds\right)^{1/2} \left(\int_{0}^{t\wedge \tau_{N}} \|V^{\epsilon}(s)\|^{2}ds\right)^{1/2}\right]\\
&\leq& 2 K_1 M^\frac{1}{2} \left(\mathbb{E}\int_{0}^{t\wedge \tau_{N}}\|V^{\epsilon}(s)\|^{2}ds\right)^\frac{1}{2} \left(\mathbb{E}\left[\sup_{t\in[0,T]}(1+\|u^\epsilon_{h_{\epsilon}}\|)^2\right]\right)^\frac{1}{2}\\
&\leq& C(M, K_1,T)\left(\mathbb{E}\int_{0}^{t\wedge \tau_{N}} \|V^{\epsilon}(s)\|^{2}ds\right)^\frac{1}{2}.
\end{eqnarray*}
Similarly, using the bounds (\ref{2.12}), (\ref{2.13}),
\begin{eqnarray*}
|I_{2}| &\leq& 2K_2\sqrt{M}\left(\mathbb{E}\int_{0}^{t\wedge \tau_{N}} \|V^{\epsilon}(s)\|^{2}ds\right)^\frac{1}{2} \left(\mathbb{E}\left[\sup_{s\in [0,T]}\|V^\epsilon(s)\|^2\right]\right)^\frac{1}{2}\\
&\leq& C(K_2, M, T) \left(\mathbb{E}\int_{0}^{t\wedge \tau_{N}} \|V^{\epsilon}(s)\|^{2}ds\right)^\frac{1}{2}.
\end{eqnarray*}
Moreover,
\begin{eqnarray*}
|I_{3}| &\leq& 2\sqrt{\epsilon} \mathbb{E}\left[\int_{0}^{t\wedge \tau_{N}}K_{1}(1+\|u_{h_{\epsilon}}^{\epsilon}(s)\|^{2})\|V^{\epsilon}(s)\|^{2}ds\right]^{1/2}\\
&\leq& 2 K_1 \sqrt{\epsilon T}\mathbb{E}\left[\sup_{s\in[0,T]}\|V^{\epsilon}(s)\|(1+\|u_{h_{\epsilon}}^{\epsilon}(s)\|)\right]\\
&\leq &C(K_1,T)\sqrt{\epsilon}.
\end{eqnarray*}

Combining these estimates, we conclude,
\begin{eqnarray}\label{4.22}
\mathbb{E}\left[\sup_{s\in[0, t\wedge \tau_{N}]}\|V^{\epsilon}(s)\|^{2}\right]\leq C(K_1,T)\sqrt{\epsilon}+C( K_1, K_2, T, M)\left(\mathbb{E}\int_{0}^{t\wedge \tau_{N}} \|V^{\epsilon}(s)\|^{2}ds\right)^\frac{1}{2}.
\end{eqnarray}
Finally, the second term in \eqref{4.22} converges to zero as $\epsilon \rightarrow 0$ by the Vitali convergence theorem leading to,
\begin{eqnarray*}
\lim_{\epsilon\rightarrow 0}\mathbb{E}\left[\sup_{ s\in [0, t\wedge \tau_{N}]}\|V^{\epsilon}(s)\|^{2}\right]=0,
\end{eqnarray*}
then letting $N\rightarrow \infty$, we achieve the desired result.
\end{proof}

\begin{proposition}
For any positive constant $M$, the set,
\begin{eqnarray}
K_{M}= \left\{\mathcal{G}^{0}\left(\int_{0}^{\cdot}h(s)ds\right): h\in \mathcal{S}_{M}(H_0)\right\},
\end{eqnarray}
is compact in $\mathcal{X}$.
\end{proposition}

\begin{proof}
It suffices to show that $K_{M}$ is sequentially compact. Let $\{h_{\epsilon}\}_{\epsilon >0}$ be a sequence in $\mathcal{S}_{M}(H_0)$ with the corresponding solutions $\{u_{h_{\epsilon}}\}_{\epsilon >0}$ to \eqref{original}. Note that $u_{h_{\epsilon}}$ is uniformly bounded in $\mathcal{C}([0,T]; H)\cap L^p(0,T;V)$ and it is not difficult to obtain that $u_{h_{\epsilon}}$ is uniformly bounded in $W^{1,2}(0,T; V')$ using a similar argument as in lemma \ref{lem4.1}. Then, the Aubin-Lions compactness embedding lemma (stated in Appendix),
\begin{eqnarray*}
 L^2(0,T;V)\cap W^{1,2}(0,T; V')\hookrightarrow L^{2}(0,T;H),
\end{eqnarray*}
implies that there exists a subsequence of $u_{h_{\epsilon}}$ still denoted by $u_{h_{\epsilon}}$, and a function $u_{h}$, such that $u_{h_{\epsilon}}\rightarrow u_h$ in $L^{2}(0,T; H)$ as $\epsilon\rightarrow 0$. It can be shown that $u_h$ is the solution corresponding to equation (\ref{controlled}) using the fact that $\{h_{\epsilon}\}_{\epsilon>0}$  has a subsequence still denoted here as $\{h_{\epsilon}\}_{\epsilon>0}$, which converges in distribution to an $h\in \mathcal{S}_{M}(H_0)$.

Let $V^{\epsilon}(t):= u_{h_{\epsilon}}(t)-u_{h}(t)$ given by,
\begin{eqnarray*}
V^{\epsilon}(t) &=& -i\int_{0}^{t} AV^{\epsilon}(s)ds - \lambda \int_{0}^{t}f(u_{h_{\epsilon}}(s))-f(u_{h}(s))ds\\&& + \int_{0}^{t} g(s, u_{h_{\epsilon}}(s))(h_{\epsilon}(s)-h(s))ds
+ \int_{0}^{t}(g(s, u_{h_{\epsilon}}(s))-g(s, u_{h}(s)))h(s)ds.
\end{eqnarray*}
Taking the inner product with $V^\epsilon$, then the supremum on time interval $[0, t]$, yields,
\begin{eqnarray*}
\sup_{s\in[0, t]} \|V^{\epsilon}(s)\|^{2} &\leq& 2\text{Im} \int_{0}^{t} \|\nabla V^{\epsilon}(s)\|^{2}ds - 2\lambda\text{Re} \int_{0}^{t} \left(f(u_{h_{\epsilon}}(s))-f(u_{h}(s)), V^{\epsilon}(s)\right)ds\\
&& + 2\left|{\rm Re}\int_{0}^{t}\left(g(s, u_{h_{\epsilon}}(s))\left(h_{\epsilon}(s)-h(s)\right), V^{\epsilon}(s)\right)ds\right| \\
&&+ 2\left|{\rm Re}\int_{0}^{t}\left(\left(g(s, u_{h_{\epsilon}}(s))-g(s, u_{h}(s))\right)h(s),V^{\epsilon}(s)\right)ds\right|.
\end{eqnarray*}
Conditions (\ref{2.6}), (\ref{2.8}) and the Cauchy-Schwarz inequality lead to,
\begin{eqnarray}\label{4.25}
 &&\quad\left|{\rm Re}\int_{0}^{t} \left(g(s, u_{h_{\epsilon}}(s))\left(h_{\epsilon}(s)-h(s)\right), V^{\epsilon}(s)\right)ds\right|\nonumber\\ &&\leq
  K_1\int_{0}^{t}\|V^{\epsilon}(s)\|\|h_{\epsilon}(s)-h(s)\|_{H_0}(1+\|u_{h_{\epsilon}}(s)\|)ds\nonumber\\ &&\leq
  K_1 \left[\sup_{s\in [0,t]}\left(1+\|u_{h_{\epsilon}}(s)\|\right)\right]\left(\int_0^t\|h_{\epsilon}(s)-h(s)\|_{H_0}^2ds\right)^\frac{1}{2}
  \left(\int_0^t\|V^\epsilon(s)\|^2ds\right)^\frac{1}{2}\nonumber\\ &&\leq C(K_1, T, M)\left(\int_0^t\|V^\epsilon(s)\|^2ds\right)^\frac{1}{2},
\end{eqnarray}
and
\begin{eqnarray}\label{4.26}
&&\left|{\rm Re}\int_{0}^{t}\left(\left(g(s, u_{h_{\epsilon}}(s))-g(s, u_{h}(s))\right)h(s),V^{\epsilon}(s)\right)ds\right|\nonumber\\ &\leq&
K_2\int_{0}^{t}\|V^{\epsilon}(s)\|^2\|h(s)\|_{H_0}ds\nonumber\\ &\leq& K_2 \left[\sup_{s\in [0,t]}\|V^\epsilon(s)\|\right]\left(\int_0^t\|h(s)\|_{H_0}^2ds\right)^\frac{1}{2}
  \left(\int_0^t\|V^\epsilon(s)\|^2ds\right)^\frac{1}{2}\nonumber\\ &\leq& C(K_2, T, M)\left(\int_0^t\|V^\epsilon(s)\|^2ds\right)^\frac{1}{2}.
\end{eqnarray}
Combining estimates (\ref{4.25}) and (\ref{4.26}), and using property in (\ref{2.3}), we conclude,
\begin{eqnarray}
\sup_{s\in [0, t]} \|V^{\epsilon}(s)\|^{2}\leq C(K_1, K_2, T, M)\left(\int_0^t\|V^\epsilon(s)\|^2ds\right)^\frac{1}{2},
\end{eqnarray}
for all $t\in [0,T]$. Hence, the strong convergence $V^\epsilon\rightarrow 0$ in $L^2(0,T;H)$ leads to the compactness of the solution in $\mathcal{X}$.
\end{proof}
{\bf Proof of Theorem 2.2}.  Combining Propositions 4.1 and 4.2, we obtain the result of Theorem 2.2 according to Theorem 4.1.

\section{Exit problem}
\setcounter{equation}{0}

As an application of the LDP, we consider the exit problem and make the remark that since large deviations was proved in space $\mathcal{C}([0,T];H)$ only, then the exit problem is limited to this space. Recall the exit time $\tau^{\epsilon}$ given by \eqref{tau} and notice that,
\begin{eqnarray*}
\mathbb{P}(\tau^{\epsilon}>T_{0}) = \mathbb{P}\left(\sup_{t\in [0, T_{0}]}\|u^{\epsilon}(t)\|^{2}\leq r\right).
\end{eqnarray*}
By the It\^{o} formula we proceed as follows,
\begin{eqnarray*}
\|u^{\epsilon}(t)\|^{2}&=& \|u_0\|^{2} -2\text{Im} \int_{0}^{t} \|\nabla u^{\epsilon}(s)\|^2ds -2\lambda \text{Re} \int_{0}^{t} (f(u^{\epsilon}(s)),u^{\epsilon}(s))ds\\
&&+ 2\sqrt{\epsilon} \text{Re} \int_{0}^{t} (g(s, u^{\epsilon}(s))dW, u^{\epsilon}(s)) + \epsilon \int_{0}^{t} \|g(s,u^{\epsilon}(s))\|_{L_{2}(H_0,H)}^{2}ds.
\end{eqnarray*}
Now by taking the supremum on time up to time $T_{0}$, using condition (\ref{2.6}) and property in (\ref{2.3}), we have,
\begin{eqnarray*}
&&\sup_{t\in [0, T_{0}]} \|u^{\epsilon}(t)\|^{2} \\ &&\leq \|u_0\|^{2} + 2\sqrt{\epsilon} \sup_{t\in [0, T_{0}]} \text{Re} \int_{0}^{t} \left(g(s, u^{\epsilon}(s))dW, u^{\epsilon}(s)\right) +\epsilon K_{1}T_{0}
 + \epsilon K_{1}T_{0}\sup_{t\in [0, T_{0}]} \|u^{\epsilon}(t)\|^{2},
\end{eqnarray*}
arriving at,
\begin{equation*}
\sup_{t\in [0, T_{0}]} \|u^{\epsilon}(t)\|^{2}\leq \frac{1}{1-\epsilon K_{1}T_{0}} \left(\|u_0\|^{2} + 2 \sqrt{\epsilon} \sup_{t\in [0, T_{0}]} \text{Re}\int_{0}^{t} \left(g(s, u^{\epsilon}(s))dW, u^{\epsilon}(s)\right) + \epsilon K_{1}T_{0}\right).
\end{equation*}
Thus, we obtain,
\begin{eqnarray*}
&&\mathbb{P}(\tau^{\epsilon}\leq T_{0})= \mathbb{P}\left(\sup_{t\in [0, T_{0}]}\|u^{\epsilon}(t)\|^{2} >r\right)\\
&&\leq \mathbb{P}\bigg(\|u_0\|^{2} + 2 \sqrt{\epsilon} \sup_{t\in [0, T_{0}]} \text{Re}\int_{0}^{t}\left(g(s, u^{\epsilon}(s))dW, u^{\epsilon}(s)\right)+ \epsilon K_{1}T_{0} > r \left(1-\epsilon K_{1}T_{0}\right)\bigg),
\end{eqnarray*}
where Doob's inequality may be applied to yield,
\begin{eqnarray*}
&&P(\tau^{\epsilon}\leq T_{0})\\&\leq& \mathbb{P}\left(2\sqrt{\epsilon} \sup_{t\in [0, T_{0}]}\text{Re}\int_{0}^{t}\left(g(s, u^{\epsilon}(s))dW, u^{\epsilon}(s)\right) > r(1-\epsilon K_{1}T_{0}) - \|u_0\|^{2} - \epsilon K_{1}T_{0}\right)\\
&\leq& \frac{2\sqrt{\epsilon}}{r\left(1-\epsilon K_{1}T_{0}\right)-\|u_0\|^{2} - \epsilon K_{1}T_{0}}\mathbb{E}\left[\int_{0}^{T_{0}}  \|g(s,u^{\epsilon}(s))\|_{L_{2}(H_0, H)}^{2}\|u^{\epsilon}(s)\|^{2}ds\right]^\frac{1}{2}\\
&\leq& \frac{2\sqrt{\epsilon K_{1}}}{r\left(1-\epsilon K_{1}T_{0}\right)-\|u_0\|^{2} -\epsilon K_{1}T_{0}} \mathbb{E} \left[\int_{0}^{T_{0}} (1+\|u^{\epsilon}(s)\|^{2})\|u^{\epsilon}(s)\|^{2}ds\right]^\frac{1}{2}\\
&\leq& \frac{2 \sqrt{\epsilon K_{1}}C(T_0, u_0)}{r\left(1-\epsilon K_{1}T_{0}\right)- \|u_0\|^{2} - \epsilon K_{1}T_{0}}.
\end{eqnarray*}
Therefore, we obtain the upperbound given in \eqref{exit}. Observe that the above upperbound complies what is expected in the physical sense. More precisely, the likelihood of the exit time to occur before a given time, $T_{0}$ becomes negligible as the radius of the domain becomes arbitrary large or as the noise is set to diminish to zero by letting $\epsilon$ tend to zero. For the lowerbound, the large deviation bound \eqref{open} may be used to obtain for the open set $D_{T_{0}}^{c}$,
\begin{eqnarray*}
\liminf_{\epsilon \rightarrow 0} \epsilon \log \mathbb{P}\left(\sup_{t\in [0, T_{0}]}\|u^{\epsilon}(t)\|^{2}>r\right)= \liminf_{\epsilon \rightarrow 0} \epsilon \log \mathbb{P}\left(u^\epsilon\in D_{T_{0}}^c\right)\geq - \inf_{v\in D_{T_{0}}^{c}} I(v).
\end{eqnarray*}
Thus, we have for a fixed $\delta >0$, there exists some $\epsilon_{0}$ such that for all $\epsilon \in (0,\epsilon_{0})$,
\begin{eqnarray*}
\mathbb{P}(\tau^{\epsilon}\leq T_{0}) \geq \exp\left(-\frac{1}{\epsilon}\left(\inf_{v\in D_{T_{0}}^{c}}I(v)+ \delta\right)\right).
\end{eqnarray*}
In addition, the upperbound \eqref{open} may be used to obtain,
\begin{eqnarray*}
\mathbb{P}(\tau^{\epsilon}>T_0) &\leq& \mathbb{P}\left(\sup_{t\in [0,T_{0}]} \|u^{\epsilon}(t)\|^{2}\leq r\right)
\leq\exp\left(-\frac{1}{\epsilon} \left(\inf_{v\in D_{T_{0}}} I(v)-\delta\right) \right),
\end{eqnarray*}
for a fixed $\delta >0$ and for all $\epsilon \in (0,\epsilon_{1})$ for some $\epsilon_{1}>0$. This together with the strong Markov property of $u^{\epsilon}(t)$ can be used to apply an inductive argument given in \cite{exit} to arrive at,
\begin{equation*}
\mathbb{P}(\tau^{\epsilon} >k) \leq \left(\sup_{u^\epsilon \in D_{1}} \mathbb{P}_{u^{\epsilon}(t)} (\tau^{\epsilon} >1)\right)^{k}.
\end{equation*}
Hence, we have,
\begin{eqnarray*}
\mathbb{E}(\tau^{\epsilon}) &\leq& \sum_{k=0}^{\infty} \mathbb{P}(\tau^{\epsilon}\geq k) \leq \sum_{k=0}^{\infty} \left(\sup_{u^{\epsilon}(t) \in D_{1} } \mathbb{P}_{u^{\epsilon}(t)} (\tau^{\epsilon}>1)\right)^{k}\\
&=& \frac{1}{1- \exp\left(-\frac{1}{\epsilon} \left(\inf_{v\in D_{1}} I(v) - \delta\right)\right)},
\end{eqnarray*}
obtaining the mean exit time.

\section{Appendix}

Here we collect some lemmas that are used in the paper.

\begin{lemma} {\rm (The Aubin-Lions lemma)} Suppose that $X_{1}\subset X_{0}\subset X_{2}$ are Banach spaces and $X_{1}$ and $X_{2}$ are reflexive and the embedding of $X_{1}$ into $X_{0}$ is compact.
Then for any $1<p<\infty,~ 0<\alpha<1$, the embedding,
\begin{equation*}
L^{p}(0,T;X_{1})\cap W^{\alpha,p}(0,T;X_{2})\hookrightarrow L^{p}(0,T;X_{0}),
\end{equation*}
is compact.
\end{lemma}
\begin{theorem}{\rm (The Vitali convergence theorem)} \label{6.2} Let $p\geq 1$, $\{X_n\}_{n\geq 1}\in L^p$ and $X_n\rightarrow X$ in probability. Then, the following are equivalent,\\
{\rm (1)} $X_n\rightarrow X$;\\
{\rm (2)} the sequence $|X_n|^p$ is uniformly integrable;\\
{\rm (3)} $\mathbb{E}|X_n|^p\rightarrow \mathbb{E}|X|^p$.
\end{theorem}

\begin{theorem} {\rm (The Skorokhod representative theorem)} Let $X$ be a Polish space. For an arbitrary sequence of probability measures $\{\nu_n\}_{n\geq 1}$ on $\mathcal{B}(X)$ weakly convergence to a probability measure $\nu$, there exists a probability space $(\Omega, \mathcal{F}, \mathbb{P})$ and a sequence of random variables $u_n, u$ with laws $\nu_n$, $\nu$, respectively and $u_n\rightarrow u$, $\mathbb{P}$ a.s. as $n\rightarrow \infty$.
\end{theorem}

\section*{Acknowledgments}
The authors thank professor Dehua Wang for useful discussions. Z. Qiu's research was supported by the China Scholarship Council under grant No.201806160015.

\end{document}